\documentclass[12pt]{article}
\usepackage{amsmath}
\usepackage{amsthm}
\usepackage{amssymb}
\newcommand{\C}{\ensuremath{\mathbb{C}}}
\newcommand{\R}{\ensuremath{\mathbb{R}}}

\usepackage[utf8]{inputenc}
\usepackage{amsfonts}
\usepackage[margin=1in]{geometry}

\usepackage{comment}

\usepackage{thmtools}

\usepackage{soul}

\usepackage{hyperref}
\usepackage{url}

\usepackage{tikz}				%
\usetikzlibrary{cd}				%
\usetikzlibrary{math}			%
\usetikzlibrary{arrows.meta,
				decorations.markings}
\usepackage{xifthen}
\usepackage{etoolbox}

\usepackage{longtable}
\usepackage{array}
\usepackage{tabularx}
\usepackage{chngcntr}

\usepackage{nicematrix}

\usepackage{soul}

\usepackage{authblk}
\usepackage{basic-commands}

\allowdisplaybreaks

\newcommand{\createPos}[1]{
	\ifthenelse{\cnttest{337}{<}{#1}  \OR \cnttest{#1}{<}{ 23}}{\def\pos{right}}{%
	\ifthenelse{\cnttest{ 22}{<}{#1} \AND \cnttest{#1}{<}{ 68}}{\def\pos{above right}}{%
	\ifthenelse{\cnttest{ 67}{<}{#1} \AND \cnttest{#1}{<}{113}}{\def\pos{above}}{%
	\ifthenelse{\cnttest{112}{<}{#1} \AND \cnttest{#1}{<}{158}}{\def\pos{above left}}{%
	\ifthenelse{\cnttest{157}{<}{#1} \AND \cnttest{#1}{<}{203}}{\def\pos{left}}{%
	\ifthenelse{\cnttest{202}{<}{#1} \AND \cnttest{#1}{<}{248}}{\def\pos{below left}}{%
	\ifthenelse{\cnttest{247}{<}{#1} \AND \cnttest{#1}{<}{293}}{\def\pos{below}}{%
	\ifthenelse{\cnttest{292}{<}{#1} \AND \cnttest{#1}{<}{338}}{\def\pos{below right}}{}
	}}}}}}}
}

\newcommand{\drawCircle}[6]{
\begin{tikzpicture}[scale=1.0]
	\tikzmath{
		\pt=0.045;
		\r = 1;
		\cx=0;	\cy=0;
		\t1=#1;  \px1=\cx+cos(\t1); \py1=\cy+sin(\t1);
		\t2=#2;  \px2=\cx+cos(\t2); \py2=\cy+sin(\t2);
		\t3=#3;  \px3=\cx+cos(\t3); \py3=\cy+sin(\t3);
		\s1=90;  \qx1=\cx+cos(\s1); \qy1=\cy+sin(\s1);
		\s2=270; \qx2=\cx+cos(\s2); \qy2=\cy+sin(\s2);
	}
	\draw[fill, white] (\qx1, \qy1) circle [radius=\pt] node[above]{$x$};
	\draw[fill, white] (\qx2, \qy2) circle [radius=\pt] node[below]{$x$};

	\foreach \p in {(\px1,\py1), (\px2,\py2), (\px3,\py3)}
	{
		\draw[fill] \p circle [radius=\pt];
	}

	\createPos{\t1}
	\node[\pos]	at (\px1,\py1) {#4};
	
	\createPos{\t2}
	\node[\pos] at (\px2,\py2) {#5};
	
	\createPos{\t3}
	\node[\pos] at (\px3,\py3) {#6};

	\draw(\cx, \cy) circle[radius=\r];
\end{tikzpicture}
}
 
\title{The topology of Curvature Sets of $\Sp^1$}

\author[ ]{Peter Eastwood\thanks{peter\_eastwood@alumni.brown.edu}}
\author[ ]{Anna M. Ellison\thanks{annamellison@gmail.com}}
\author[1]{Mario G\'{o}mez\thanks{gomezflores.1@buckeyemail.osu.edu}}
\author[1]{Facundo M\'{e}moli\thanks{facundo.memoli@gmail.com}}

\affil[1]{Department of Mathematics. The Ohio State University.}

\begin{document}
\maketitle

\begin{abstract}
	For $n \geq 2$, the $n$-th curvature set of a metric space $X$ is the set consisting of all $n$-by-$n$ distance matrices of $n$ points sampled from $X$. Curvature sets can be regarded as a geometric analogue of configuration spaces. 
 In this paper we carry out a geometric and topological study of the curvature sets of the unit circle $\Sp^1$ equipped with the geodesic metric. Via an inductive argument we compute the homology groups of all  curvature sets of $\Sp^1$. We also construct an abstract simplicial complex, called the $n$-th State Complex, whose geometric realization is homeomorphic to the $n$-th Curvature Set of $\Sp^1$. %
\end{abstract}
\tableofcontents

\section{Introduction.}
\label{sec:intro}
\counterwithout{theorem}{section}

In this paper, we study the topological and geometric structure of the curvature sets of the unit circle $\Sp^1$ equipped with its geodesic metric. For $n \geq 2$, the \emph{$n$-th curvature set} of a metric space $(X, d_X)$ is the set
\begin{equation*}
	\Kn_n(X) := \Big\{M \in \R^{n \times n}: \exists \ \{x_1, \dots, x_n\} \subset X %
	\text{ such that } M_{ij} = d_X(x_i,x_j) \Big\}.
\end{equation*}
The \emph{distance matrix map} $D_n:X^n \to \Kn_n(X)$ is defined by setting $M := D_n(x_1, \dots, x_n)$ to be the matrix such that $M_{ij} := d_X(x_i,x_j)$. In other words, $\Kn_n(X) = \operatorname{Im}(D_n)$ is the set of $n$-by-$n$ distance matrices generated by points of $X$. Similar objects have been studied before. Perhaps the most well known are \emph{configuration spaces}, the sets
\begin{equation*}
	\operatorname{Conf}_n(X) := \{(x_1, \dots, x_n) \in X: \text{ where } x_i \neq x_j \text{ if } i \neq j\},
\end{equation*}
where $X$ is a topological space. In words, the configuration space $\operatorname{Conf}_n(X)$ is the set of $n$-point tuples of distinct points in $X$. Another similar space is the \emph{Ran space} $\exp_n(X)$, which is the collection of nonempty subsets of $X$ of cardinality at most $n$ equipped with the quotient topology induced by the natural map $X^n \to \exp_n(X)$ defined by $(x_1, \dots, x_n) \mapsto \{x_1\} \cup \cdots \cup \{x_n\}$. There is a substantial body of research on configuration spaces \cite{confspaces_hard_spheres_morse, confspaces_homology, confspaces_lie_algebras, confspaces_introduction, confspaces_unordered_cohomology, confspaces_rational_homotopy, confspaces_surfaces_cohomology, confspaces_sphere_loop_spaces, confspaces_sphere_integral_cohomology, confspaces_varieties} which includes applications to physics \cite{confspaces_disks_infinite_strip, confspaces_hard_disks_computation} and robotics \cite{confspaces_factory, confspaces_braids_robotics}. There are also studies specifically on the Ran spaces of the circle (see the work of Bott \cite{ran-space-S1-old} and Tuffley \cite{ran-space-S1}), including the characterization of the homotopy types of $\exp_n(\Sp^1)$ as a sphere $\Sp^{n}$ if $n$ is odd or $\Sp^{n-1}$ if $n$ is even \cite[Theorem 4]{ran-space-S1}. In fact, the complement of $\exp_{k-2}(\Sp^1)$ in $\exp_{k}(\Sp^1)$ has the homotopy type of a $(k-1,k)$-torus knot complement \cite[Theorem 6]{ran-space-S1}.\\
\indent The study of configuration and Ran spaces has a strong geometric flavor and even allows for very concrete results such as Theorems 4 through 7 in \cite{ran-space-S1}, which, for example, calculate the degree of the  map $\exp_k f: \exp_k(\Sp^1) \to \exp_k(\Sp^1)$ induced by a given map $f:\Sp^1 \to \Sp^1$ and also the action induced by $\exp_{k-1}(\Sp^1) \hookrightarrow \exp_{k}(\Sp^1)$ in $H_{2k-1}$. Tuffley also computed the rational homology of $\exp_k(\Sp^2)$ \cite[Theorem 1]{ran-space-surfaces} and the top two homology groups of $\exp_k(\Sigma)$ for a closed surface $\Sigma$ \cite[Theorem 3]{ran-space-surfaces}. As for curvature sets, \cite{mem12} provides descriptions of $\Kn_3(\Sp^1)$ and $\Kn_3(\Sp^2)$ which we now recall. Notice that a 3-by-3 distance matrix is symmetric and is determined by the entries $d_{12}, d_{23}, d_{31}$, so we can visualize $\Kn_3(X)$ as a subset of $\R^3$. In the case of $\Sp^1$, a set of three points in $\Sp^1$ may or may not be contained in a semicircle. If it is, one of the distances is the sum of the other two. As a consequence, $\Kn_3(\Sp^1)$ contains the three 2-simplices determined by the linear relations $d_{\sigma(1), \sigma(3)} = d_{\sigma(1), \sigma(2)}+d_{\sigma(2), \sigma(3)}$, where $\sigma$ is one of three permutations in $S_3$, and where  $0 \leq d_{ij} \leq \pi$ for $1 \leq i,j \leq 3$. If the set is not contained in a semicircle, then $d_{12}+d_{23}+d_{31}=2\pi$. This adds one more 2-simplex to $\Kn_3(\Sp^1)$, so that $\Kn_3(\Sphere{1})$ is homeomorphic to the boundary of the 3-simplex with vertices $(0,0,0)$, $(\pi, \pi, 0)$, $(\pi, 0, \pi)$, and $(0, \pi, \pi)$; see Figure \ref{fig:K3_S1}. In fact, each of the 2-simplices that we described is the convex hull of its vertices; we give more details in Example \ref{ex:K3_S1}. As for $\Kn_3(\Sp^2)$, any 3 points in general position lie on a plane whose intersection with $\Sp^2$ is a circle $C$ with radius at most 1. While the restriction of the metric of $\Sp^2$ to $C$ is not always geodesic, we do have that $d_{ij}$ is less than or equal to the length of the shortest path in $C$ that joins $x_i$ and $x_j$. With a bit more work, we see that $\Kn_3(\Sp^2)$ actually equals the union of all  $\lambda \cdot \Kn_3(\Sp^1)$, where $0 < \lambda \leq 1$, and that, therefore,  $\Kn_3(\Sp^2)$ is the convex hull of $\Kn_3(\Sp^1)$.

\begin{figure}[h]
	\centering
\begin{tikzpicture}[scale=1.7,
	z={(-0.3cm,-0.3cm)}, %
	line join=round, line cap=round %
]
\draw[dashed] (0,0,0) -- (1,1,0);

\draw (0,0,0) -- (0,1,1);
\draw (0,0,0) -- (1,0,1);

\draw (0,1,1) -- (1,1,0) -- (1,0,1) -- cycle;

\draw[->] (0,0,0) -- (0,0,1.5) node[below left] {\small $d_{23}$};
\draw[dotted, ->] (0,0,0) -- (0,1.3,0) node[above] {\small $d_{31}$};
\draw[dotted, ->] (0,0,0) -- (1.3,0,0) node[right] {\small $d_{12}$};

\coordinate (p1) at (-0.2,1,1);
\coordinate (p2) at (1,0,1);
\coordinate (p3) at (1,1,0);

\coordinate (q1) at (-1.67,1,1);
\coordinate (q2) at ( 2.60,0,1);
\coordinate (q3) at ( 1.8,1,-2.1);

\node (Tp1) [label={[name=Lp1] above: \small $(0,\pi,\pi)$}] at (p1) {};
\node (Tp2) [label={[name=Lp2] below: \small $(\pi,\pi,0)$}] at (p2) {};
\node (Tp3) [label={[name=Lp3]: \small $(\pi,0,\pi)$}] at (p3) {};

\node [label={[name=Lq1] above:
				\small $D_1 =
				\left(
				\begin{smallmatrix}
					  0 &   0 & \pi \\
					  0 &   0 & \pi \\
					\pi & \pi &   0
				\end{smallmatrix}
				\right)$
			}
	  ] at (q1) {};

\node [label={[name=Lq2] below:
				\small $D_2 =
				\left(
				\begin{smallmatrix}
					  0 & \pi & \pi \\
					\pi &   0 &   0 \\
					\pi &   0 &   0
				\end{smallmatrix}
				\right)$
			}
	  ] at (q2) {};

\node [label={[name=Lq3] below:
	\small $D_3 =
	\left(
	\begin{smallmatrix}
		  0 & \pi &   0 \\
		\pi &   0 & \pi \\
		  0 & \pi &   0
	\end{smallmatrix}
	\right)$
}
] at (q3) {};

\draw[|->, semithick] (Lq1) -- (Lp1);
\draw[|->, semithick] (Lq2) -- (Lp2);
\draw[|->, semithick] (Lq3) -- (Lp3);

\draw[<-|, thick] (0.35,0.58,1) -- (0,0.55,1.5) node[anchor=east, align=left] {
	\small
	$\alpha D_1 + (1-\alpha)D_2$
	\\[0.25em]
	$=\left(
	\begin{smallmatrix}
		0 				& (1-\alpha)\pi	& \pi \\
		(1-\alpha)\pi	& 0				& \alpha\pi \\
		\pi				& \alpha\pi		& 0
	\end{smallmatrix}
	\right)$
	};

\def\pt{0.03};
\draw[fill] (0,1,1) circle[radius=\pt];
\draw[fill] (1,0,1) circle[radius=\pt];
\draw[fill] (1,1,0) circle[radius=\pt];

\draw[fill] (0.4,0.6,1) circle[radius=\pt];
\end{tikzpicture}  	\caption{A visualization of $\Kn_3(\Sp^1)$ in $\R^3$. The three simplices incident on the origin correspond to configurations $x_1, x_2, x_3 \in \Sp^1$ contained in a semicircle; all other configurations correspond to the remaining 2-simplex. Every simplex is the convex hull of its vertices. The set $\Kn_3(\Sp^2)$ is the full 3-simplex.}
	\label{fig:K3_S1}
\end{figure}
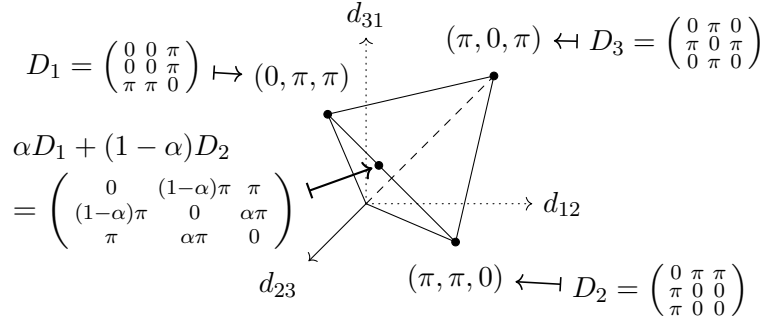

Perhaps the fundamental result regarding curvature sets, though, was given by Gromov. He proved that curvature sets are a complete invariant of compact metric spaces, which means that two compact metric spaces $X$ and $Y$ are isometric if and only if $\Kn_n(X) = \Kn_n(Y)$ for all $n \geq 2$ \cite[Section 3.27]{gro99}. After that, \cite{mem12} proved that curvature sets are stable under the Gromov-Hausdorff distance, that is,
\begin{equation*}
	\dH^{\R^{n \times n}}(\Kn_n(X), \Kn_n(Y)) \leq 2 \cdot \dGH(X,Y),
\end{equation*}
where $\dH^{\R^{n \times n}}$ is the Hausdorff distance between subsets of $\R^{n \times n}$ equipped with the $\ell^\infty$ metric and $\dGH$ is the Gromov-Hausdorff distance between metric spaces. That paper also has polynomial-time lower bounds for the Gromov-Hausdorff distance that can be computed from curvature sets.\\
\indent Other than that, however, it doesn't seem that the geometric or topological study of curvature sets -- which we regard as a geometry enriched configuration space -- has received much attention and this paper is a first step in that direction. We give several characterizations and simplicial models of the curvature sets of $\Sp^1$. The main model is as a quotient of the torus $\T^{n}$ by the diagonal action of the group $O(2)$ of isometries of the circle. We also characterize the homotopy type of several subsets of $\Kn_{n}(\Sp^1)$ and their intersection. In particular, $\Kn_n(\Sp^1)$ is the union of two open subsets each of which is homotopy equivalent to $\Kn_{n-1}(\Sp^1)$ -- a fact  which allows us to use the Mayer-Vietoris sequence and and an induction argument for computing the homology groups of $\Kn_n(\Sp^1)$.
\begin{restatable}{theorem}{HomologyKn}\label{thm:main}
	\[
	H_m\big(\Kn_n(\Sp^1),\mathbb{Z}\big) = 
	\begin{cases}
		\mathbb{Z}& m=0,\\
		\mathbb{Z}^{\binom{n-1}{m}} \oplus (\mathbb{Z} / 2\mathbb{Z})^{\sum_{i=0}^{n-m-2} \binom{n-1}{i}} & m\ \mathrm{even},\ \mathrm{and}\ 2\leq m \leq n-1, \\
		0 &\mathrm{else}.
	\end{cases}
	\]
\end{restatable}
Note that when $n=3$, the homology groups coincide with those of $\Sp^2$, which is consistent with the fact that $\Kn_3(\Sp^1)\cong \Sp^2$ as described above. Note also that when $n>3$, the number of generators of the torsion subgroups can range from $1$ to $2^{n-1}$ (minus a quadratic polynomial in $n$).\\

\indent We also construct an abstract simplicial complex that generalizes the intuition in Figure \ref{fig:K3_S1}. We call it the \emph{State Complex} $\St_n(\Sp^1)$. In Figure \ref{fig:K3_S1}, $\Kn_3(\Sp^1)$ is shown to have four 2-simplices. Three of these correspond to sets of 3 points that are contained in a semicircle, and the simplex is determined by the ordering of the points. The last simplex corresponds to those sets that are not contained in a semicircle. Intuitively, $\St_n(\Sp^1)$ contains two types of $(n-1)$-simplices: some correspond to an ordering of $n$ points in a semicircle, while others give the ways that such an ordered set might fail to be contained in a semicircle (e.g. when every point but the 4th lies on the same semicircle). Moreover, the vertices of $\St_n(\Sp^1)$ are points $(x_1, \dots, x_n) \in \T^n$ such that every $x_i$ is equal or antipodal to $x_1$ and, for each simplex $A \in \St_n(\Sp^1)$, $\Kn_n(\Sp^1)$ contains the convex hull of the distance matrices of the vertices of $A$. For instance, the vertices $D_1$ and $D_2$ of $\Kn_3(\Sp^1)$, as depicted  in Figure \ref{fig:K3_S1},  correspond to the distance matrices of  tuples of the form $(x,x,-x)$ and $(x,-x,-x)$, respectively, for  $x \in \Sp^1$, and the edge between $D_1$ and $D_2$ is the convex hull of $\{D_1, D_2\}$.  In fact, we prove the following:
\begin{restatable}{theorem}{GeometricRealizationStn}\label{thm:geometric_realization_intro}
    $\displaystyle
	\Kn_n(\Sp^1) = \bigcup_{A \in \St_n(\Sp^1)} \conv(D_n(A)) \cong |\St_n(\Sp^1)|.
    $
\end{restatable}
\noindent We prove this Theorem at the end of Section \ref{subsec:geometric_realization}.\\
\indent We conclude the paper by identifying a precise relationship between $\Kn_n(\Sp^1)$ and the $n$-th \emph{elliptope}, a basic object studied in semidefinite optimization. Our results about the homology groups of $\Kn_n(\Sp^1)$ provide a measure of complexity of the elliptopes.

\subsubsection*{Organization.}
We start by setting up the relevant notation in Section \ref{sec:notation}. After that, the paper has two main sections: Section \ref{sec:homology} on the homology groups of $\Kn_n(\Sp^1)$ and Section \ref{sec:state_complex} on the state complex $\St_n(\Sp^1)$. Section \ref{sec:homology} shows that the distance matrix map $D_n:\T^n \to \Kn_n(\Sp^1)$ induces the quotient topology induced by the diagonal action of $O(2)$ on $\T^n$. In Section \ref{sec:Kn^(t)}, we study the homotopy type of certain subsets of $\Kn_n(\Sp^1)$ that we later use to compute the homology groups $H_m(\Kn_n(\Sp^1))$ in Section \ref{sec:homology_computation}. Before the final computation, we give $\Kn_n(\Sp^1)$ a CW structure and compute the action induced by several maps in Section \ref{sec:cellular_structure} in chain groups and homology.

In the second part, we start by defining cluster structures in Section \ref{subsec:cluster_structures}. A cluster structure is a function that determines the order of points of $\mathbf{x} \in \T^n$ when normalized to a semicircle and indicates which components of $\mathbf{x}$ do not lie in that semicircle. Given a cluster structure $\mathfrak{c}$, we define a function $\Phi_{\mathfrak{c}}:\Delta_{m-1} \to \T^n$ such that the set $\Im(D_n \circ \Phi_{\mathfrak{c}})$ is the convex hull of a set of $m$ points. In Section \ref{subsec:state_complex}, we work out the technical details to verify that $\St_n(\Sp^1)$ is an abstract simplicial complex. We prove $\Kn_n(\Sp^1) \cong |\St_n(\Sp^1)|$ in Section \ref{subsec:geometric_realization}. We also find the number of simplices and the Euler characteristic of $\St_n(\Sp^1)$ in Section \ref{sec:combinatorics_St_n}. We conclude by mentioning the connections with elliptopes and possible future directions in Section \ref{sec:elliptopes}. We also include an Appendix with an explicit calculation of $\St_3(\Sp^1)$ to illustrate the constructions in Section \ref{sec:state_complex}.

\subsubsection*{Acknowledgements.}
This work was funded by the grants NSF-IIS-1901360, NSF-CCF-1740761, NSF-CCF-1839358, and BSF-2020124. This paper emerged from a project ran during the 2017  Summer at ICERM 2017 resarch program \url{https://icerm.brown.edu/summerug/2017/}. The authors would also like to thank Dave Anderson for the helpful discussions during which the connection of curvature sets with elliptopes arose.

\counterwithin{theorem}{section}

\section{Notation.}
\label{sec:notation}

\indent Given a topological space $X$, a boldface letter $\mathbf{x}$ represents a vector $\mathbf{x} := (x_1, \dots, x_n) \in X^n$. Given a function $f:X \to X$, we also use $f$ to denote the product map $X^n \to X^n$ that sends $\mathbf{x} = (x_1, \dots, x_n)$ to $f(\mathbf{x}) := (f(x_1), \dots, f(x_n))$. We use $\R^{n \times n}$ to denote the space of $n$-by-$n$ real matrices with the usual topology. Given $M \in \R^{n \times n}$, we write $M_{i*} \in \R^m$ for the $i$-th row of and $M_{*j} \in \R^n$ for the $j$-th column of $M$.\\
\indent For concreteness, we view $\Sp^1$ as the unit circle in $\C$. We write $\arg:\C \setminus \{0\} \to [0,2 \pi)$ for the argument function. We follow the usual convention that $\arg(z)$ increases in the anticlockwise direction. We denote the geodesic metric on $\Sp^1$ simply as $d$, where
\begin{equation*}
	\textstyle
	d(x,y) := \min\left(\arg(\frac{x}{y}), \arg(\frac{y}{x}) \right).
\end{equation*}
Let $\rho:\Sp^1 \to \Sp^1$ be $\rho(z) := z^{-1}$, the reflection along the real axis in $\C$. We denote the group of orthogonal transformations of $\Sp^1$ as $O(2)$. The $n$-torus is the product space $\T^n := (\Sp^1)^n$. For the rest of the paper, the map $D_n$ will denote the map $D_n:\T^n \to \Kn_n(\Sp^1)$ such that $M = D_n(\mathbf{x})$ is the matrix where $M_{ij} := d(x_i,x_j)$. 

\section{Homology groups of $\Kn_n(\Sp^1)$.}
\label{sec:homology}

The objective of this section is to compute the homology groups of $\Kn_n(\Sp^1)$. We start by showing that the defining map $D_n:\T^n \to \Kn_n(\Sp^1)$ is a quotient map under the action of $O(2)$.

\subsection{$\Kn_n(\Sp^1)$ as a quotient of  $\T^n$.}
\label{sec:quotient_of_torus}
\indent Observe that the distance matrices $D_n(\mathbf{x})$ and $D_n(\tau(\mathbf{x}))$ are equal for any isometry $\tau \in O(2)$. We show the converse in this section: if $D_n(\mathbf{x}) = D_n(\mathbf{y})$, then there exists $\tau \in O(2)$ such that $\mathbf{y} = \tau(\mathbf{x})$.

\begin{lemma}
	\label{lemma:3point_isometry}
	Let $x_1, x_2, x_3 \in \Sp^1$ and $x_1',x_2',x_3' \in \Sp^1$ such that $d(x_i,x_j) = d(x_i',x_j')$. Suppose $d(x_1,x_2) \neq 0, \pi$. Let $\tau \in O(2)$ such that $\tau(x_1) = x_1'$ and $\tau(x_2) = x_2'$. Then $\tau(x_3) = x_3'$.
\end{lemma}
\begin{proof}
	First, observe that any $\tau \in O(2)$ is determined by its action on a basis of $\C$ as an $\R$-vector space. Since $d(x_1,x_2) \neq 0, \pi$, the set $\{x_1, x_2\}$ is a basis, so $\tau$ is determined by $\tau(x_1)=x_1'$ and $\tau(x_2)=x_2'$. Now, observe that $d(x_i,x_j)=d(x_i',x_j')$ if and only if $\frac{x_i}{x_j} = \frac{x_i'}{x_j'}$ or $\frac{x_i}{x_j} = \frac{x_j'}{x_i'}$. If $\frac{x_2}{x_1} = \frac{x_2'}{x_1'}$, then $\tau(z) := \frac{x_1'}{x_1} \cdot z$ is unique such that $\tau(x_1) = x_1'$ and $\tau(x_2) = x_2'$. We claim that $\frac{x_3'}{x_2'} = \frac{x_3}{x_2}$ and $\frac{x_3'}{x_1'} = \frac{x_3}{x_1}$. In fact, the equation $\frac{x_2}{x_1} \frac{x_3}{x_2} \frac{x_1}{x_3} = 1 = \frac{x_2'}{x_1'} \frac{x_3'}{x_2'} \frac{x_1'}{x_3'}$, which simplifies to $\frac{x_3}{x_2} \frac{x_1}{x_3} = \frac{x_3'}{x_2'} \frac{x_1'}{x_3'}$, implies that $\frac{x_3'}{x_2'} = \frac{x_3}{x_2}$ if and only if $\frac{x_3'}{x_1'} = \frac{x_3}{x_1}$. If we had the opposite equalities $\frac{x_3'}{x_2'} = \frac{x_2}{x_3}$ and $\frac{x_3'}{x_1'} = \frac{x_1}{x_2}$, then the equation
	\begin{equation*}
		\frac{x_3}{x_2} \frac{x_1}{x_3} = \frac{x_3'}{x_2'} \frac{x_1'}{x_3'} = \frac{x_2}{x_3} \frac{x_3}{x_1}
	\end{equation*}
	gives $(\frac{x_1}{x_2})^2 = 1$. However, $x_2 \neq \pm x_1$, so we must have $\frac{x_3'}{x_2'} = \frac{x_3}{x_2}$ and $\frac{x_3'}{x_1'} = \frac{x_3}{x_1}$. Therefore, $\tau(x_3) = \frac{x_1'}{x_1} \cdot x_3 = x_1' \cdot \frac{x_3'}{x_1'} = x_3'$.\\
	\indent If we had $\frac{x_2}{x_1} = \frac{x_1'}{x_2'}$ instead, the conditions of the lemma still apply to $1/x_i'$ and $\rho \circ \tau$ because $\rho$ is an isometry. The paragraph above implies that $\rho \circ \tau(x_3) = 1/x_3'$, which gives $\tau(x_3) = x_3'$.
\end{proof}

\begin{lemma}
	\label{lemma:dm_up_to_O2}
	For all $\mathbf{x},\mathbf{y} \in \T^n$ such that $D_{n}(\mathbf{x}) = D_{n}(\mathbf{y})$, there exists some $\tau \in O(2)$ such that $\tau(\mathbf{x}) = \mathbf{y}$.
\end{lemma}
\begin{proof}
	We first consider the case of $n=2$. $D_2(\mathbf{x})=D_2(\mathbf{y})$ implies $d(x_1,x_2) = d(y_1,y_2)$ and, with that, $\frac{x_1}{x_2} = \frac{y_1}{y_2}$ or $\frac{x_1}{x_2} = \frac{y_2}{y_1}$. In the first case, $\tau(z) := \frac{y_1}{x_1} \cdot z$ satisfies $\tau(\mathbf{x}) = \tau(\mathbf{y})$. If $\frac{x_1}{x_2} = \frac{y_2}{y_1}$, define $\tau(z) := \frac{1/y_1}{x_1} z$. This time, $\rho \circ \tau$ satisfies $\rho \circ \tau(\mathbf{x}) = \rho \circ \tau(\mathbf{y})$.\\
	\indent Now, let $n>2$. If all the entries of $D_{n}(\mathbf{x}) = D_n(\mathbf{y})$ are either $0$ or $\pi$, we must have $x_i = \varepsilon_i x_1$ and $y_i = \varepsilon_i y_1$, where $\varepsilon_i \in \{+1, -1\}$. Then $\tau(z) := \frac{y_1}{x_1} \cdot z$ satisfies $\tau(\mathbf{x}) = \mathbf{y}$. Otherwise, pick $i_1, i_2$ such that $d(x_{i_1}, x_{i_2}) \ne 0, \pi$. By the case $n=2$, there exists $\tau \in O(2)$ such that $\tau(x_{i_1}) = y_{i_1}$ and $\tau(x_{i_2}) = y_{i_2}$. Then by Lemma \ref{lemma:3point_isometry}, for any $i \neq i_1, i_2$ we have $\tau(x_i) = y_i$.
\end{proof}

With these lemmas, we can show that the quotient by $D_n$ is induced by $O(2)$.
\begin{prop}
	\label{prop:quotient_of_big_torus}
	$D_n:\T^n \to \Kn_n(\Sp^1)$ is the open quotient map under the diagonal action of $O(2)$. Hence $\Kn_n(\Sp^1) \cong \T^n / O(2)$.
\end{prop}
\begin{proof}
	Since every coordinate function $d(x_i,x_j)$ is continuous, $D_n$ is as well and, by Lemma \ref{lemma:dm_up_to_O2}, $D_n(\mathbf{x}) = D_n(\mathbf{y})$ if and only if there exists $\tau \in O(2)$ such that $\tau(\mathbf{x}) = \mathbf{y}$. Since $D_n$ is also surjective, it is a quotient map and $\Kn_n(\Sp^1) \cong \T^n / O(2)$. The quotient is induced by a group action, so $D_n$ is open.
\end{proof}

In fact, we can see $\Kn_n(\Sp^1)$ as a quotient of $\T^{n-1}$ instead by modding out the translations in $O(2)$. For concreteness, define $i_{n} : \T^{n-1} \to \T^{n}$ by $i_n(x_1, \ldots ,x_{n-1}) := (1, x_1, \ldots , x_{n-1})$, and let $\widehat{D}_n := D_n \circ i_{n}: \T^{n-1} \to \Kn_n(\Sp^1)$.
\begin{prop}
	\label{prop:quotient_of_torus}
	For $\mathbf{x}, \mathbf{y} \in \T^{n-1}$, $\widehat{D}_n(\mathbf{x}) = \widehat{D}_n(\mathbf{y})$ implies either $\mathbf{y} = \mathbf{x}$ or $\mathbf{y} = \rho(\mathbf{x})$. As a consequence, $\widehat{D}_n$ is an open quotient map and $\Kn_n(\Sp^1) \cong \T^{n-1}/\langle \rho \rangle$.
\end{prop}
\begin{proof}
	Since $D_n \circ i_{n}(\mathbf{x}) = D_n \circ i_{n}(\mathbf{y})$, Lemma \ref{lemma:dm_up_to_O2} gives a $\tau \in O(2)$ such that $\tau(i_{n}(\mathbf{x})) = i_{n}(\mathbf{y})$. The first coordinates of $i_{n}(\mathbf{x})$ and $i_{n}(\mathbf{y})$ are $+1$, so we must have $\tau = \mathrm{id}_{\Sp^1}$ or $\tau = \rho$. Since $\widehat{D}_n$ is surjective, it is the open quotient map $\T^{n-1} \to \T^{n-1}/\langle \rho \rangle$. Hence, $\Kn_n(\Sp^1) \cong \T^{n-1}/\langle \rho \rangle$.
\end{proof}

\subsection{Homotopy types of the subsets $\Kn_n^{(t)}(\Sp^1)$.}
\label{sec:Kn^(t)}
\indent Given $t \in [0,\pi]$, define $\Kn_n^{(t)}(\Sp^1)$ as the subspace of $\Kn_{n}(\Sp^1)$ of elements $M$ such that $M_{1n} = t$. Clearly, $\Kn_n(\Sp^1)$ is the union of the two subsets $\Kn_n(\Sp^1) \setminus \Kn_n^{(0)}(\Sp^1)$ and $\Kn_n(\Sp^1) \setminus \Kn_n^{(\pi)}(\Sp^1)$. The objective of this section is to determine the homotopy type of these sets and that of their intersection in order to compute the homology groups of $\Kn_n(\Sp^1)$ using the Mayer-Vietoris sequence in Section \ref{sec:homology_computation}.

\begin{lemma}
	\label{lemma:K_n^(t)}
	If $t = 0$ or $\pi$, $\Kn_n^{(t)}(\Sp^1)$ is homeomorphic to $\Kn_{n-1}(\Sp^1)$.
\end{lemma}
\begin{proof}
	Let $t=0$ first, and consider the diagram:
	\begin{equation*}
		\begin{tikzcd}
			\T^{n-2} \arrow{r}{\widehat{D}_{n-1}} \arrow[d, "i"', "\cong"] & \Kn_{n-1}(\Sp^1) \arrow[dashed]{d}{\iota_0} \\
			\T^{n-2} \times \{+1\} \arrow{r}{\widehat{D}_n} & \Kn_{n}^{(0)}(\Sp^1).
		\end{tikzcd}
	\end{equation*}
	Notice that $\T^{n-2} \times \{+1\} = \widehat{D}_n^{-1}(\Kn_n^{(0)}(\Sp^1)) = \widehat{D}_n^{-1} \circ \widehat{D}_n(\T^{n-2} \times \{+1\})$, that is, the closed set $\T^{n-2} \times \{+1\} \subset \T^{n-1}$ is saturated with respect to $\widehat{D}_n$. Thus, the restriction of $\widehat{D}_n$ to $\T^{n-2} \times \{+1\} \to \Kn_n^{(0)}(\Sp^1)$ is a quotient map.\\
	\indent Now, let $\mathbf{x}, \mathbf{y} \in \T^{n-2}$. We claim that $\widehat{D}_{n-1}(\mathbf{x}) = \widehat{D}_{n-1}(\mathbf{y})$ if and only if $\widehat{D}_n \circ i(\mathbf{x}) = \widehat{D}_n \circ i(\mathbf{y})$. Indeed, $\widehat{D}_{n-1}(\mathbf{x}) = \widehat{D}_{n-1}(\mathbf{y})$ if and only if there exists $\tau \in \langle \rho \rangle$ such that $\tau(\mathbf{x}) = \mathbf{y}$ by Proposition \ref{prop:quotient_of_torus}. Since $\tau$ fixes $+1$, the last condition is also equivalent $\tau(i(\mathbf{x})) = i(\mathbf{y})$ which, in turn, is equivalent to $\widehat{D}_n \circ i(\mathbf{x}) = \widehat{D}_n \circ i(\mathbf{y})$ by Proposition \ref{prop:quotient_of_torus}. Then, the universal property of the quotient map $\widehat{D}_{n-1}$ produces a unique continuous map $\iota_0:\Kn_{n-1}(\Sp^1) \to \Kn_n^{(0)}(\Sp^1)$ that makes the diagram above commute. Similarly, the universal property of $\widehat{D}_n$ applied to $\widehat{D}_{n-1} \circ i_n^{-1}$ produces a map $\Kn_n^{(0)}(\Sp^1) \to \Kn_{n-1}(\Sp^1)$. This map has to be $\iota_0^{-1}$ because $\widehat{D}_{n}(\mathbf{x}) = \widehat{D}_{n}(\mathbf{y})$ if and only if $\widehat{D}_{n-1} \circ i^{-1}(\mathbf{x}) = \widehat{D}_{n-1} \circ i^{-1}(\mathbf{y})$ for any $\mathbf{x}, \mathbf{y} \in \T^{n-2} \times \{+1\}$. Thus, $\iota_0$ is a homeomorphism. The analogous argument shows that $\Kn_n^{(\pi)}(\Sp^1)$ is homeomorphic to $\Kn_{n-1}(\Sp^1)$ because $\langle \rho \rangle$ also fixes $-1 \in \Sp^1$.
\end{proof}

\begin{lemma}
	\label{lemma:retract_Kn^(t)_complement}
	If $t=0$ or $t=\pi$, then $\Kn_n^{(t)}(\Sp^1)$ is a deformation retract of $\Kn_n(\Sp^1) \setminus \Kn_n^{(\pi-t)}(\Sp^1)$.
\end{lemma}
\begin{proof}
	Let $H_{-1} := \Sp^1 \setminus \{+1\}$, the set of points in $\Sp^1$ of the form $e^{ir}$ for $0 < r < 2\pi$. Define $\xi:H_{-1} \times [0,1] \to \{-1\}$ to be the deformation retraction given by $\xi(e^{ir},t) := e^{i[(1-t)r+\pi t]}$, where $0 < r < 2\pi$. Consider $I := \text{Id}_{\T^{n-2}}$ and the deformation retraction $I \times \xi$ from $\T^{n-2} \times H_{-1}$ onto $\T^{n-2} \times \{-1\}$. Observe that $\T^{n-2} \times H_{-1}$ is an open subset of $\T^{n-1}$ that is saturated with respect to $\widehat{D}_n$, so the restriction of $\widehat{D}_n$ to $\T^{n-2} \times H_{-1} \to \Kn_n(\Sp^1) \setminus \Kn_n^{(0)}(\Sp^1)$ is a quotient map. Also, observe that $\rho(e^{ir}) = e^{i(2\pi-r)}$, so
	\begin{equation*}
		\xi(\rho(e^{ir}), t) = \xi(e^{i(2\pi-r)},t) = e^{i[(1-t)(2\pi-r)+\pi t]} = e^{i[2\pi-(1-t)r-\pi t]} = \rho \circ \xi(e^{ir},t).
	\end{equation*}
	Hence, the map $I \times \xi$ commutes with $\rho \times \text{Id}_{[0,1]}$ and thus descends to a homotopy defined on the quotient $\T^{n-2} \times H_{-1} / \langle \rho \rangle \cong \Kn_n(\Sp^1) \setminus \Kn_n^{(0)}(\Sp^1)$. Composing with $\widehat{D}_n$ gives a deformation retraction from $\widehat{D}_n \circ (I \times \xi)(\T^{n-2} \times H_{-1}, 0) = \widehat{D}_n(\T^{n-2} \times H_{-1}) = \Kn_n(\Sp^1) \setminus \Kn_n^{(0)}(\Sp^1)$ onto $\widehat{D}_n \circ (I \times H)(\T^{n-2} \times H_{-1}, 1) = \widehat{D}_n(\T^{n-2} \times \{-1\}) = \Kn_n^{(\pi)}(\Sp^1)$. The case of $\Kn_n(\Sp^1) \setminus \Kn_n^{(\pi)}(\Sp^1)$ and $\Kn_n^{(0)}(\Sp^1)$is analogous.
\end{proof}

\begin{corollary}
	\label{cor:Kn^(t)_complement_is_Kn}
	If $t=0$ or $t=\pi$, then $\Kn_n(\Sp^1) \setminus \Kn_n^{(t)}(\Sp^1)$ is homotopy equivalent to $\Kn_{n-1}(\Sp^1)$.
\end{corollary}
\begin{proof}
	By Lemma \ref{lemma:retract_Kn^(t)_complement}, $\Kn_n(\Sp^1) \setminus \Kn_n^{(t)}(\Sp^1)$ is homotopy equivalent to $\Kn_n^{(\pi-t)}(\Sp^1)$ which, by Lemma \ref{lemma:K_n^(t)}, is homeomorphic to $\Kn_{n-1}(\Sp^1)$.
\end{proof}

\begin{lemma}
	\label{lemma:torus-interval}
	$\left(\Kn_n(\Sp^1) \setminus \Kn_n^{(\pi)}(\Sp^1)\right) \cap \left(\Kn_n(\Sp^1) \setminus \Kn_n^{(0)}(\Sp^1)\right) \cong \T^{n-2} \times (0, \pi) \simeq \T^{n-2}$. In particular, $\Kn_n^{(t)}(\Sp^1)$ is homeomorphic to $\T^{n-2}$ whenever $t \neq 0, \pi$.
\end{lemma}
\begin{proof}
	\indent Define $H^+ := \{e^{it}: 0 < t < \pi\}$, $H^- := \rho(H^+)$, and $H := H^+ \cup H^-$. Let $\widetilde{K}_n := \left(\Kn_n(\Sp^1) \setminus \Kn_n^{(\pi)}(\Sp^1)\right) \cap \left(\Kn_n(\Sp^1) \setminus \Kn_n^{(0)}(\Sp^1)\right)$. Observe that $\widehat{D}_n^{-1}(\widetilde{K}_n) = \T^{n-2} \times H$. Since $\rho(\T^{n-2} \times H^+) = \T^{n-2} \times H^-$, Proposition \ref{prop:quotient_of_torus} implies that $\widehat{D}_n^{-1}(M) \cap (\T^{n-2} \times H^+)$ has exactly one element for every $M \in \widetilde{K}_n$. Hence, the restriction of $\widehat{D}_n$ to $\T^{n-2} \times H^+ \to \widetilde{K}_n$ is bijective. Since $\widehat{D}_n$ is continuous and open by Proposition \ref{prop:quotient_of_torus}, it is a homeomorphism from $\T^{n-2} \times H^+ \cong \T^{n-2} \times (0, \pi)$ to $\widetilde{K}_n$. Furthermore, since $H^+$ is contractible, $\widetilde{K}_n$ is homotopy equivalent to $\T^{n-2}$. In particular, the restriction $\T^{n-2} \times \{t\} \cong \T^{n-2} \xrightarrow{\widehat{D}_n} \Kn_n^{(t)}(\Sp^1)$ is a homeomorphism for any specific $0 < t <\pi$.
\end{proof}

\begin{remark}
	\label{rmk:double_mapping_cylinder}
	The proof of Lemma \ref{lemma:torus-interval} cannot be extended to show that $\Kn_n^{(t)}(\Sp^1) \cong \T^{n-2}$ when $t=0,\pi$ because the restriction of $\widehat{D}_n$ to $\T^{n-2} \times \{t\} \to \Kn_n^{(t)}(\Sp^1)$ is 2-to-1. However, combining Lemmas \ref{lemma:K_n^(t)} and \ref{lemma:torus-interval} gives yet another representation of $\Kn_n(\Sp^1)$ as the double mapping cylinder $\T^{n-2} \times [0,\pi] / \sim$ where $(\mathbf{x}, i) \sim (\rho(\mathbf{x}), i)$ for $i=0,1$.
\end{remark}

\subsection{A cellular structure for $\Kn_n(\Sp^1)$ and its chain groups.}
\label{sec:cellular_structure}
In this section, we occasionally write $\Kn_n = \Kn_n(\Sp^1)$ to reduce notational overload. We start by giving $\Kn_n(\Sp^1)$ a CW-complex structure induced from the quotient $\Kn_n(\Sp^1) = \T^{n-1}/\langle \rho \rangle$ in Proposition \ref{prop:quotient_of_torus}. Give $\Sp^1$ the CW-complex structure with two $0$-dimensional cells $v_1$ and $v_2$ and two $1$-dimensional cells $e_1$ and $e_2$ with boundary $v_2 - v_1$. Consider the characteristic maps into $\Sp^1 \subset \C$ that send $v_1 \mapsto +1$, $v_2 \mapsto -1$, $e_1 \mapsto (t \in [0,1] \mapsto e^{+it})$, and $e_2 \mapsto (t \in [0,1] \mapsto e^{-it})$. Under this structure, the reflection $\rho$ interchanges $e_1$ and $e_2$ while preserving their orientation -- in other words, $\rho$ is a cellular map. This CW-complex induces a product CW structure on the torus $\T^{n-1} = (\Sp^1)^{n-1}$ along with a cellular action of $\rho$. We know from the K\"unneth Theorem that the cellular chain groups and the homology of $\T^{n-1}$ satisfy
\begin{align*}
	C_m(\T^{n-1}) &= \bigoplus C_{i_1}(\Sp^1) \otimes \cdots \otimes C_{i_{n-1}}(\Sp^1)\\
	H_m(\T^{n-1}) &= \bigoplus H_{i_1}(\Sp^1) \otimes \cdots \otimes H_{i_{n-1}}(\Sp^1),
\end{align*}
where each sum runs over all indices $0 \leq i_j \leq 1$ such that $i_1 + \cdots i_{n-1} = m$. Both groups are equipped with an action of the symmetric group $S_{n-1}$.\\
\indent Since $\rho$ is a cellular map, Proposition \ref{prop:quotient_of_torus} implies that $\Kn_n(\Sp^1)$ inherits a CW structure from $\T^{n-1}$ wherein a $k$-cell in $\Kn_n(\Sp^1)$ is the equivalence class $[a] \in C_k(\T^{n-1})/\langle \rho \rangle$ of a $k$-cell from $\T^{n-1}$. Under this identification, the quotient map $\widehat{D}_n$ is cellular, so it induces the following structure on the chain groups:
\begin{equation*}
	C_m(\Kn_n(\Sp^1)) \cong C_m(\T^{n-1})/(a \sim \rho_*(a)).
\end{equation*}
By abuse of notation, we write $a$ both for an element of $C_m(\T^{n-1})$ and its equivalence class $(\widehat{D}_n)_*(a) \in C_m(\Kn_n(\Sp^1))$. However, we explicitly write $a \sim b$ to indicate that $a$ and $b$ in $C_m(\T^{n-1})$ represent the same class in $C_m(\Kn_n(\Sp^1))$.\\
\indent Additionally, the action of $S_{n-1}$ on $\T^{n-1}$ commutes with $\rho$, so it also descends to an action of $S_{n-1}$ on $\Kn_n$ and its chain complex. In particular, we can use the action of $S_{n-1}$ to describe the generators of $H_m(\T^{n-1})$ in terms of the 1-cycle $e_1-e_2$ that generates $H_1(\Sp^1)$. Going forward, the notation $(e_1-e_2)^{\otimes m} \otimes v_1^{\otimes (n-m-1)}$ represents the $m$ cycle $(e_1-e_2) \otimes \cdots \otimes (e_1-e_2) \otimes v_1 \otimes \cdots \otimes v_1$, where $(e_1-e_2)$ is repeated $m$ times and $v_1$, $n-m-1$ times. Then $L_m^{n-1} := \{\sigma \cdot ((e_1-e_2)^{\otimes m} \otimes v_1^{\otimes (n-m-1)}): \sigma \in S_{n-1}\}$ is a generating set for $H_m(\T^{n-1})$, where $\sigma \cdot (b_{1} \otimes \cdots \otimes b_{n-1}) = b_{\sigma(1)} \otimes \cdots \otimes b_{\sigma(n-1)}$.

\begin{prop}\label{prop:mayer-vietoris}
	Let $A := \Kn_n(\Sp^1) \setminus \Kn_n^{(\pi)}(\Sp^1)$ and $B := \Kn_n(\Sp^1) \setminus \Kn_n^{(0)}(\Sp^1)$. The homology groups of $\Kn_n$ obey the following Mayer-Vietoris sequence:
	\begin{align*}
		\cdots \xrightarrow{\partial_*} \, H_{m}(\mathbb{T}^{n-2}) \, \xrightarrow{(i_*,j_*)} \, H_{m}(\Kn_{n-1}) \oplus H_{m}(\Kn_{n-1}) \, \xrightarrow{k_* - l_*} \, H_{m}(\Kn_{n}) \xrightarrow{\partial_*} \, H_{m-1}(\mathbb{T}^{n-2}) \rightarrow \cdots
	\end{align*}
	where $i_*$, $j_*$, $k_*$, and $l_*$ are the maps induced by the inclusions $i: A \cap B \hookrightarrow A$, $j: A \cap B \hookrightarrow B$, $k: A \hookrightarrow \Kn_n$, $l: B \hookrightarrow \Kn_n$.
\end{prop}

\begin{proof}
	The conditions to use the Mayer-Vietoris sequence are satisfied because the sets $A$ and $B$ are open as they are the image of the open sets $\T^{n-2} \times (\Sp^1 \setminus \{-1\})$ and $\T^{n-2} \times (\Sp^1 \setminus \{+1\})$, respectively, under the open map $\widehat{D}_n$. Also, $\Kn_n = A \cup B$, $A \cong B \simeq \Kn_{n-1}$ by Corollary \ref{cor:Kn^(t)_complement_is_Kn}, and $A \cap B \simeq \T^{n-2}$ by Lemma \ref{lemma:torus-interval}.
\end{proof}

Thanks to the Lemmas in the last section, we can give an explicit description of the action of the maps in the Mayer-Vietoris sequence. For the inclusion $i:A \cap B \hookrightarrow A$, Lemmas \ref{lemma:torus-interval} and \ref{lemma:retract_Kn^(t)_complement} say that $A \cap B \simeq \T^{n-2}$ and $A \simeq \Kn_n^{(0)}(\Sp^1)$, so we can replace the map $i:A \cap B \hookrightarrow A$ with $i:\T^{n-2} \xrightarrow{\cong} \T^{n-2} \times \{+1\} \xrightarrow{\widehat{D}_n} \Kn_n^{(0)}(\Sp^1)$, that is, $i(\mathbf{x}) := \widehat{D}_n(\mathbf{x} \times (+1))$. Similarly, $j:A \cap B \hookrightarrow B$ is replaced with $j:\T^{n-2} \to \Kn_n^{(\pi)}(\Sp^1)$ defined by $j(\mathbf{x}) := \widehat{D}_n(\mathbf{x} \times (-1))$. As for $k:A \hookrightarrow \Kn_n$ and $l:B \hookrightarrow \Kn_n$, these are equivalent to the inclusions $k:\Kn_n^{(0)} \simeq A \hookrightarrow \Kn_n$ and $l:\Kn_n^{(\pi)} \simeq B \hookrightarrow \Kn_n$ by Lemma \ref{lemma:retract_Kn^(t)_complement}. The actions induced by these maps in the chain groups are as follows:
\begin{itemize}
	\item $i_*(a) = a \otimes v_1 \in C_m(\Kn_{n-1})$ for $a \in C_m(\T^{n-2})$,
	\item $j_*(a) = a \otimes v_2 \in C_m(\Kn_{n-1})$ for $a \in C_m(\T^{n-2})$,
	\item $k_*(a \otimes v_1) = a \otimes v_1 \in C_m(\Kn_n)$ for $a \in C_m(\Kn_{n-1})$, and
	\item $l_*(a \otimes v_2) = a \otimes v_2 \in C_m(\Kn_n)$ for $a \in C_m(\Kn_{n-1})$.
\end{itemize}

To compute the homology groups $H_m(\Kn_n)$ via the Mayer-Vietoris sequence, we need to know the action of some of the maps in the sequence. We calculate that explicitly for $i_*, j_*$ and the boundary map $\partial_*$. We start with:
\begin{lemma}\label{lemma:q_injective_or_0}
	The restriction of $(\widehat{D}_n)_*:C_m(\T^{n-1}) \to C_m(\Kn_n)$ to $L_m^{n-1}$ is 0 when $m$ is odd and induces multiplication by $2$ when $m$ is even.
\end{lemma}
\begin{proof}
	Let $a_0 := (e_1-e_2)^{\otimes m} \otimes \cdots \otimes v_1^{\otimes n-m-1} \in L_m^{n-1}$. Then
	\begin{align*}
		& (\widehat{D}_n)_*(a_0) = (e_1-e_2) \otimes (e_1-e_2)^{\otimes(m-1)} \otimes v_1 ^{\otimes(n-m-1)} \\
		&= e_1 \otimes (e_1-e_2)^{\otimes(m-1)} \otimes v_1 ^{\otimes(n-m-1)} - e_2 \otimes (e_1-e_2)^{\otimes(m-1)} \otimes v_1 ^{\otimes(n-m-1)} \\
		&\sim e_1 \otimes (e_1-e_2)^{\otimes(m-1)} \otimes v_1 ^{\otimes(n-m-1)} - e_1 \otimes (e_2-e_1)^{\otimes(m-1)} \otimes v_1 ^{\otimes(n-m-1)} \\
		&= (1+(-1)^{m})e_1 \otimes (e_1-e_2)^{\otimes(m-1)} \otimes v_1 ^{\otimes(n-m-1)}.
	\end{align*}
	Hence, $(\widehat{D}_n)_*(a_0) = 0$ if $m$ is odd and $(\widehat{D}_n)_*(a_0) = 2 \cdot e_1 \otimes (e_1-e_2)^{\otimes(m-1)} \otimes v_1 ^{\otimes(n-m-1)}$ if $m$ is even. Since the action of $\sigma \in S_{n-1}$ commutes with $(\widehat{D}_n)_*$, the analogous result holds for the rest of the terms $\sigma \cdot a \in L_m^{n-1}$.
\end{proof}

As a consequence of the previous lemma, the action of the maps $i_*, j_*: H_m(\T^{n-2}) \to H_m(\Kn_{n-1})$ is trivial if $m$ is odd, but we can't yet say that it is injective if $m$ is even. To show that, we first study the boundary map $\partial_*:H_{m+1}(\Kn_n) \to H_m(\T^{n-2})$. We need to subdivide the CW structure on $\Sp^1$ as shown in Figure \ref{fig:new_circle} in order to evaluate $\partial_*$. Notice that this subdivision is still compatible with the action of $\rho$.
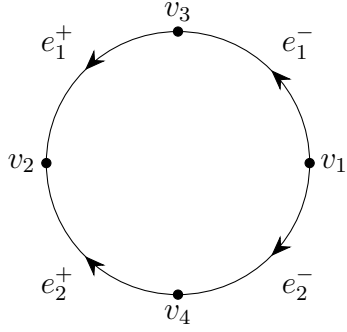
\begin{figure}[h]
	\centering
\begin{tikzpicture}[scale=1.75]
	\tikzmath{
		\pt=0.035;
		\r = 1;
		\cx=0;	\cy=0;
		\t1=  0;  \px1=\cx+cos(\t1); \py1=\cy+sin(\t1);
		\t2=180;  \px2=\cx+cos(\t2); \py2=\cy+sin(\t2);
		\t3= 90;  \px3=\cx+cos(\t3); \py3=\cy+sin(\t3);
		\t4=270;  \px4=\cx+cos(\t4); \py4=\cy+sin(\t4);
	}

	\tikzset{
		mid arrow/.style={postaction={decorate,decoration={
					markings,
					mark=at position .5 with {\arrow[#1]{Stealth[length=8]}}
				}}}
 	}

	\foreach \p in {(\px1,\py1), (\px2,\py2), (\px3,\py3), (\px4,\py4)}
	{
		\draw[fill] \p circle [radius=\pt];
	}

	\node[right]	at (\px1,\py1) {$v_1$};
	\node[ left]	at (\px2,\py2) {$v_2$};
	\node[above]	at (\px3,\py3) {$v_3$};
	\node[below]	at (\px4,\py4) {$v_4$};

	\draw[mid arrow](\cx+\r, \cy)
		arc[start angle=0,end angle=90,radius=\r]
		node[midway, above right]{$e_1^-$};
	
	\draw[mid arrow](\cx, \cy+\r)
		arc[start angle=90,end angle=180,radius=\r]
		node[midway, above left]{$e_1^+$};
	
	\draw[mid arrow](\cx, \cy-\r)
		arc[start angle=270,end angle=180,radius=\r]
		node[midway, below left]{$e_2^+$};
	
	\draw[mid arrow](\cx+\r, \cy)
		arc[start angle=0,end angle=-90,radius=\r]
		node[midway, below right]{$e_2^-$};
\end{tikzpicture} 	\caption{Subdivision of the CW structure of the circle.}
	\label{fig:new_circle}
\end{figure}

\begin{lemma}
	\label{lemma:q_injective}
	If $m \geq 2$ is even, then the set $\frac{1}{2}(\widehat{D}_n)_*(L_{m}^{n-1})$ is linearly independent in $H_m(\Kn_n)$.
\end{lemma}
\begin{proof}
	Let us write $\partial_*$ for the boundary map in the Mayer-Vietoris sequence of Proposition \ref{prop:mayer-vietoris} and $\partial$ for the boundary map in a chain complex (which complex will be clear from context). Recall that to find $\partial_*(c)$ for $c \in H_{m}(\Kn_n)$, we have to write $c = a+b$, where $a \in C_m(A) = C_m(\Kn_{n-1}^{(0)}(\Sp^1))$ and $b \in C_m(B) = C_m(\Kn_{n-1}^{(\pi)}(\Sp^1))$ such that $\partial(a) = - \partial(b)$. Then we set $\partial_*(c) := \partial(a) = -\partial(b) \in H_{m-1}(A \cap B)$.\\
	\indent Let $a_0 := e_1 \otimes (e_1-e_2)^{\otimes(m-1)} \otimes v_1 ^{\otimes(n-m-1)}$. The set $\frac{1}{2}(\widehat{D}_n)_*(L_{m}^{n-1})$ is generated by elements of the form $\sigma \cdot a_0$ by Lemma \ref{lemma:q_injective_or_0}. Observe that the $i$-th component of $\sigma \cdot a_0$ is $v_1$ if and only if $\sigma(i)>m$. In particular, the last coordinate of $\sigma \cdot a_0$ is $v_1$ when $\sigma(n) > m$. In that case, $\sigma \cdot a_0 = \sigma \cdot a_0 + 0$, where $\sigma \cdot a_0 \in A, 0 \in B$, and thus, $\partial_*(\sigma \cdot a_0) = -\partial(0) = 0$. However, if $\sigma(n) \leq m$, we can write $\sigma \cdot a_0 = v_1^{\otimes k} \otimes e_1 \otimes a_0' \otimes (e_1-e_2)$, where $k := \min\{i: \sigma(i+1) \leq m\}$ and $a_0' \in L_{m-2}^{n-k-2} \subset C_{m-2}(\T^{n-k-2})$. Write $E_1$ for the set of $\sigma \cdot a_0 \in \frac{1}{2}(\widehat{D}_n)_*(L_{m}^{n-1})$ such that $\sigma(n) > m$ and $E_2$ for those $\sigma \cdot a_0$ for which $\sigma(n) \leq m$.\\
	\indent Suppose then that $\sigma \cdot a_0 \in E_2$. Since $e_1 = e_1^+ + e_1^-$ and $e_2 = e_2^+ + e_2^-$, we can write 
	\begin{equation*}
		\sigma \cdot a_0
		= v_1^{\otimes k} \otimes e_1 \otimes a_0' \otimes (e_1^+ - e_2^+)
		+ v_1^{\otimes k} \otimes e_1 \otimes a_0' \otimes (e_1^- - e_2^-),
	\end{equation*}
	where the first term is in $C_m(\Kn_n^{(0)}(\Sp^1))$ and the second, in $C_m(\Kn_n^{(\pi)}(\Sp^1))$. Then, calculating in $C_m(\Kn_{n}^{(0)}) \subset C_m(\Kn_n)$, we have
	\begin{align}
		&\partial\left(v_1^{\otimes k} \otimes e_1 \otimes a_0' \otimes (e_1^+ - e_2^+)\right) \nonumber\\
		&= v_1^{\otimes k} \otimes \partial(e_1) \otimes a_0' \otimes (e_1^+ - e_2^+) - v_1^{\otimes k} \otimes e_1 \otimes \partial(a_0') \otimes (e_1^+ - e_2^+) + (-1)^{m-1} v_1^{\otimes k} \otimes e_1 \otimes a_0' \otimes \partial(e_1^+ - e_2^+) \nonumber\\
		&= v_1^{\otimes k} \otimes (v_2-v_1) \otimes a_0' \otimes (e_1^+ - e_2^+) - v_1^{\otimes k} \otimes e_1 \otimes 0 \otimes (e_1^+ - e_2^+) + (-1)^{m-1} v_1^{\otimes k} \otimes e_1 \otimes a_0' \otimes (v_4 - v_3) \nonumber\\
		&= v_1^{\otimes k} \otimes (v_2-v_1) \otimes a_0' \otimes (e_1^+ - e_2^+) + (-1)^{m-1} v_1^{\otimes k} \otimes e_1 \otimes a_0' \otimes (v_4 - v_3).
		\label{eq:boundary_comp}
	\end{align}
	Observe that $a_0'$ has $(m-2)$ factors $e_1-e_2$, so $\rho_*(a_0') = (-1)^{m-2} a_0'$, which equals $a_0'$ because $m$ is even. Then the first term of Equation (\ref{eq:boundary_comp}) equals
	\begin{align*}
		& v_1^{\otimes k} \otimes (v_2-v_1) \otimes a_0' \otimes (e_1^+ - e_2^+) \\
		&\sim v_1^{\otimes k} \otimes (v_2-v_1) \otimes a_0' \otimes e_1^+ - \rho_*\left(v_1^{\otimes k} \otimes (v_2-v_1) \otimes a_0' \otimes e_2^+\right) \\
		&= v_1^{\otimes k} \otimes (v_2-v_1) \otimes a_0' \otimes e_1^+ - v_1^{\otimes k} \otimes (v_2-v_1) \otimes a_0' \otimes e_1^+\\
		&= 0.
	\end{align*}
	As for the second term of Equation (\ref{eq:boundary_comp}),
	\begin{align*}
		v_1^{\otimes k} \otimes e_1 \otimes a_0' \otimes (v_4 - v_3)
		&= v_1^{\otimes k} \otimes e_1 \otimes a_0' \otimes v_4 - v_1^{\otimes k} \otimes e_1 \otimes a_0' \otimes v_3 \\
		&\sim v_1^{\otimes k} \otimes e_1 \otimes a_0' \otimes v_4 - \rho_*\left(v_1^{\otimes k} \otimes e_1 \otimes a_0' \otimes v_3 \right)\\
		&= v_1^{\otimes k} \otimes e_1 \otimes a_0' \otimes v_4 - v_1^{\otimes k} \otimes e_2 \otimes a_0' \otimes v_4\\
		&= v_1^{\otimes k} \otimes (e_1 - e_2) \otimes a_0' \otimes v_4.
	\end{align*}
	Lastly, notice that
	\begin{equation*}
		\partial\left( v_1^{\otimes k} \otimes (e_1 - e_2) \otimes a_0' \otimes e_2^- \right) = v_1^{\otimes k} \otimes (e_1 - e_2) \otimes a_0' \otimes v_4 - v_1^{\otimes k} \otimes (e_1 - e_2) \otimes a_0' \otimes v_1.
	\end{equation*}
	Putting the last three equations together shows that (\ref{eq:boundary_comp}) equals $(-1)^{m-1} v_1^{\otimes k} \otimes (e_1 - e_2) \otimes a_0' \otimes v_1$ plus a boundary in $C_{m-1}(\Kn_n^{(0)})$. Thus,
	\begin{align*}
		\partial_*([\sigma \cdot a_0]) &= \left[\partial\left(v_1^{\otimes k} \otimes e_1 \otimes a_0' \otimes (e_1^+ - e_2^+)\right)\right]\\
		&= [(-1)^{m-1} v_1^{\otimes k} \otimes (e_1 - e_2) \otimes a_0'],
	\end{align*}
	which is an element of $L_{m-1}^{n-2} \subset H_{m-1}(\T^{n-2})$.\\
	\indent Clearly, the assignments $a_0 \mapsto a_0'$ and $[\sigma \cdot a_0] \mapsto \partial_*([\sigma \cdot a_0])$ are injective. Thus, $\partial_*\left(\frac{1}{2}(\widehat{D}_n)_*(L_{m}^{n-1}) \right) = \partial_*(E_2)$ has the same amount of elements as the subset of $L_{m}^{n-1}$ whose last coordinate is $(e_1-e_2)$, which is $\binom{n-2}{m}$. However, that number is also the rank of $H_{m-1}(\T^{n-2})$, from which we conclude $\partial_*\left(\frac{1}{2}(\widehat{D}_n)_*(L_{m}^{n-1}) \right) = (-1)^{m-1} L_{m-1}^{n-2}$. Thus, a zero linear combination of elements of $E_2$ induces a zero linear combination in $H_{m-1}(\T^{n-2})$ which must be trivial because $L_{m-1}^{n-2}$ is a generating set of $H_{m-1}(\T^{n-2})$. Hence, the starting linear combination must be trivial in $H_m(\Kn_n)$ and it follows that $E_2$ is linearly independent.\\
	\indent What about the elements of $E_1$? Notice that $(\widehat{D}_{n})_*^{-1}(2 \cdot E_1) = L_{m}^{n-2} \otimes \{v_1\} \subset H_m(\T^{n-2} \times \{+1\})$. Using the notation of Lemma \ref{lemma:K_n^(t)},
	\begin{equation*}
		\textstyle
		(\iota_0^{-1})_*(E_1) = \frac{1}{2}(\iota_0^{-1} \circ \widehat{D}_{n})_*(L_{m}^{n-2} \otimes \{v_1\}) = \frac{1}{2}(\widehat{D}_{n-1} \circ i^{-1} )_*(L_{m}^{n-2} \otimes \{v_1\}) = \frac{1}{2}(\widehat{D}_{n-1} )_* (L_{m}^{n-2}).
	\end{equation*}
	Hence, an inductive argument on $n$ shows that $E_1$ is linearly independent because the set $\frac{1}{2}(\widehat{D}_{n-1} )_* (L_{m}^{n-2}) \subset H_m(\Kn_{n-1})$ already is.\\
	\indent So far we have that both $E_1$ and $E_2$ are linearly independent. To finish the proof of this lemma, observe that $\langle E_1 \rangle \cap \langle E_2 \rangle = 0$ in $H_m(\Kn_n)$ because the last coordinate in every element of $E_1$ is $v_1$, while the last coordinate in $E_2$ is $e_1-e_2$. Hence, the set $\frac{1}{2}(\widehat{D}_n)_*(L_{m}^{n-1}) = E_1 \cup E_2$ is linearly independent in $H_m(\Kn_n)$.
\end{proof}

\begin{corollary}
	\label{cor:q_injective_or_0}
	The map $(\widehat{D}_n)_*:H_m(\T^{n-1}) \to H_m(\Kn_n)$ is 0 whenever $m$ is odd and injective when $m$ is even, in which case it induces coordinate-wise multiplication by $2$.
\end{corollary}
\begin{proof}
	By Lemma \ref{lemma:q_injective_or_0}, $(\widehat{D}_n)_*$ is 0 or multiplication by 2 at the level of chains. Hence, the map in homology satisfies the same property. If additionally $m$ is even, Lemma \ref{lemma:q_injective} shows that $(\widehat{D}_n)_*$ is injective because it sends the generating set $L_{m}^{n-1}$ to the linearly independent set $(\widehat{D}_n)_*(L_{m}^{n-1})$.
\end{proof}

\subsection{The computation of $H_m(\Kn_n(\Sp^1))$.}
\label{sec:homology_computation}
Now, we have all the tools to prove our main result.

\HomologyKn*

\begin{proof}
    Recall the shorthand notation $\Kn_n = \Kn_n(\Sp^1)$. We start by giving a recursive formula to obtain the homology groups of $\Kn_n$ in terms of those of $\Kn_{n-1}$. First of all, $\Kn_2$ consists only of the edge $e_1 \sim e_2$ pointing from $v_1$ to $v_2$, so its reduced homology groups are all 0. Additionally, $H_0(\Kn_n)=\mathbb{Z}$ because $\Kn_n$ is the image of the connected space $\mathbb{T}^{n-1}$ under the continuous map $\widehat{D}_n$. As a consequence, the terms at the right end of the Mayer-Vietoris sequence (Proposition \ref{prop:mayer-vietoris})
	\begin{equation*}
		H_1(\Kn_{n}) \xrightarrow{\partial_*} \, H_{0} (\mathbb{T}^{n-2}) \rightarrow H_0(\Kn_{n-1}) \oplus H_0(\Kn_{n-1}) \, \to \, H_0(\Kn_{n}) \rightarrow 0
	\end{equation*}
	are $H_1(\Kn_{n}) \xrightarrow{\partial_*} \, \mathbb{Z} \rightarrow \mathbb{Z}^2 \, \rightarrow \, \mathbb{Z} \rightarrow 0$. The last 4 terms are part of the short-exact sequence $0 \to \mathbb{Z} \xrightarrow{\left(\begin{smallmatrix} +1 \\ +1 \end{smallmatrix}\right)} \mathbb{Z}^2 \xrightarrow{\left(\begin{smallmatrix} +1 \\ -1 \end{smallmatrix}\right)^\intercal} \mathbb{Z} \to 0$, so $\partial_* = 0$.\\
	\indent By Corollary \ref{cor:q_injective_or_0}, $i_*$ and $j_*$ are 0 in odd dimensions and induce coordinate-wise multiplication by 2 in even non-zero dimensions. Then, in
	\begin{equation*}
		H_{2k+1}(\Kn_{n}) \, \xrightarrow{\partial_*} \, H_{2k}(\mathbb{T}^{n-2}) \, \xrightarrow{(i_*,j_*)} \, H_{2k}(\Kn_{n-1}) \oplus H_{2k}(\Kn_{n-1}),
	\end{equation*}
	we have $\operatorname{Im} \partial_* = \ker (i_*, j_*) = 0$. Moving to the left in the sequence above gives
	\begin{equation*}
		H_{2k+1}(\mathbb{T}^{n-2}) \xrightarrow{0} H_{2k+1}(\Kn_{n-1}) \oplus H_{2k+1}(\Kn_{n-1}) \, \rightarrow \, H_{2k+1}(\Kn_{n}) \xrightarrow{\partial_* = 0} H_{2k}(\mathbb{T}^{n-2}).
	\end{equation*}
	Thus, $H_{2k+1}(\Kn_{n}) = H_{2k+1}(\Kn_{n-1}) \oplus H_{2k+1}(\Kn_{n-1})$. Since $H_{2k+1}(\Kn_2)=0$ for all $k \geq 0$, we obtain $H_{2k+1}(\Kn_n)=0$.\\
	\indent Now we only need to consider the even-dimensional terms in the sequence. Let's rename the maps as follows:
	\begin{equation*}
		0 \to H_{2k}(\mathbb{T}^{n-2}) \xrightarrow{a} H_{2k}(\Kn_{n-1}) \oplus H_{2k}(\Kn_{n-1}) \, \xrightarrow{b} \, H_{2k}(\Kn_{n}) \xrightarrow{c} H_{2k-1}(\mathbb{T}^{n-2}) \to 0.
	\end{equation*}
	Write $H_{2k}(\Kn_n) = \mathbb{Z}^{\beta_{n,2k}} \oplus T_{n,2k}$, where $T_{n,2k}$ is the torsion subgroup. Let $\beta_{n,m}$ be the rank of $H_m(\Kn_n)$ and $\beta_{n,m}^{(2)}$ be the number of generators of $T_{n,2k}$. Then the above sequence becomes:
	\begin{equation*}
		0 \to \mathbb{Z}^{\binom{n-2}{2k}} \xrightarrow{a} \mathbb{Z}^{2 \beta_{n-1,2k}} \oplus T_{n-1,2k}^2 \xrightarrow{b} \mathbb{Z}^{\beta_{n,2k}} \oplus T_{n,2k} \xrightarrow{c} \mathbb{Z}^{\binom{n-2}{2k-1}} \to 0.
	\end{equation*}
	Here we make several observations. First, $a(\mathbb{Z}^{\binom{n-2}{2k}}) = (2 \cdot \mathbb{Z})^{\binom{n-2}{2k}} \subset \mathbb{Z}^{2\beta_{n-1,2k}}$ because $a$ is injective. Since $b \circ a = 0$, if $\sigma$ is a generator of $H_{2k}(\mathbb{T}^{n-2})$, then $b$ maps $\frac{1}{2} a(\sigma)$ into $T_{n,2k}$. This is because $2 \cdot b(\frac{1}{2}(a(\sigma))) = b(a(\sigma)) = 0$. Also, $b$ is injective on $T_{n-1,2k}^2$ since this subgroup doesn't intersect $\operatorname{Im}(a) = \ker(b)$ except at 0, so $b(T_{n-1,2k}^2) \subset T_{n,2k}$. No other element from $H_{2k}(\Kn_{n-1})^2$ maps to $T_{n,2k}$ under $b$, so $T_{n,2k} \cong T_{n-1,2k}^2 \oplus \frac{1}{2}\operatorname{Im}(a)$. Since the elements in $\frac{1}{2} \operatorname{Im}(a)$ have order 2, by induction, $T_{n,2k}$ consists solely of copies of $\mathbb{Z}/2\mathbb{Z}$. Moreover, the coefficents $\beta_{n,2k}^{(2)}$ satisfy
	\begin{equation*}
		\textstyle
		\beta_{n,2k}^{(2)} = 2 \cdot \beta_{n-1,2k}^{(2)} + \binom{n-2}{2k}.
	\end{equation*}
	A simple induction argument permits verifying that $\beta_{n,2k}^{(2)} = \sum_{i=0}^{n-2k-2} \binom{n-2}{i}$.\\
	\indent Regarding the free terms, since $\ker(b) = \operatorname{Im}(a)$, $b$ is injective on $H_{2k}(\Kn_{n-1})/\operatorname{Im}(a)$, and in particular, $b(\mathbb{Z}^{2\beta_{n-1,2k}} / \frac{1}{2} \operatorname{Im}(a)) \subset \mathbb{Z}^{\beta_{n,2k}}$. Also, since $c$ is surjective, $$H_{2k}(\Kn_n)/\operatorname{Im}(b) \to H_{2k-1}(\mathbb{T}^{n-2})$$ is an isomorphism. In particular, $b$ restricted to $\mathbb{Z}^{2\beta_{n-1,2k}} / \frac{1}{2} \operatorname{Im}(a)$ must be an isomorphism onto a free subgroup of $\mathbb{Z}^{\beta_{n,2k}}$. These facts imply the relation
	\begin{equation*}
		\textstyle
		\beta_{n,2k} = \operatorname{rk}(\mathbb{Z}^{2\beta_{n-1,2k}} / \frac{1}{2} \operatorname{Im}(a)) + \operatorname{rk}(H_{2k-1}(\mathbb{T}^{n-2})) = 2\beta_{n-1,2k} - \binom{n-2}{2k} + \binom{n-2}{2k-1}.
	\end{equation*}
	The closed formula $\beta_{n,2k} = \binom{n-1}{2k}$ can then be obtained by induction using the relation above together with the identity $\binom{n-2}{2k-1}+\binom{n-2}{2k} = \binom{n-1}{2k}$.
\end{proof}

\section{The state complex of $\Sp^1$.}
\label{sec:state_complex}

To motivate the constructions and theorems in this section, let's revisit the calculations that yielded Figure \ref{fig:K3_S1}, i.e. the visualization of $\Kn_3(\Sp^1)$ as the boundary of a 3-simplex in $\R^3$.

\begin{example}
	\label{ex:K3_S1}
	Let $\mathbf{v}_1 = (1,-1,-1)$, $\mathbf{v}_2 = (1,-1,1)$, $\mathbf{v}_3 = (1,1,-1)$, and $\mathbf{v}_4 = (1,1,1)$. These are all the points $\mathbf{v} \in \T^3$ such that $v_1 = 1$ and $v_i = \pm 1$ for $2 \leq i \leq 3$. Define
	\begin{equation*}
		\begin{array}{cccc}
			M_1 :=
			\begin{pmatrix}
				0   & \pi & \pi \\
				\pi & 0   & 0   \\
				\pi & 0   & 0
			\end{pmatrix}
			&
			M_2 :=
			\begin{pmatrix}
				0   & \pi & 0   \\
				\pi & 0   & \pi \\
				0   & \pi & 0
			\end{pmatrix}
			&
			M_3 :=
			\begin{pmatrix}
				0   & 0   & \pi \\
				0   & 0   & \pi \\
				\pi & \pi & 0
			\end{pmatrix}
			&
			M_4 :=
			\begin{pmatrix}
				0   & 0   & 0    \\
				0   & 0   & 0    \\
				0   & 0   & 0
			\end{pmatrix}.
		\end{array}
	\end{equation*}
	Notice that $M_i = D_3(\mathbf{v}_i) \in \Kn_3(\Sp^1)$ for $1 \leq i \leq 4$. Let $\mathbf{x} \in \T^3$ be arbitrary and denote $d_{ij} := d(x_i, x_j)$. Rotate and (if necessary) reflect $\mathbf{x}$ so that $\arg(x_1) = 0$ and $\arg(x_2) \in [0, \pi]$. Let $t_i := \arg(x_i)$. Notice that $\mathbf{x}$ is contained in a semicircle if and only if $t_3 \in [0,\pi] \cup [t_2+\pi, 2\pi]$ and, in that case, one of the distances $d_{12}, d_{13}, d_{23}$ is the sum of the other two. If $t_3 \in [\pi, t_2+\pi]$, then $d_{12}+d_{23}+d_{31}=2\pi$ instead. With these observations, we can find an expression for the distances $d_{12}, d_{13}, d_{23}$ in terms of $t_2$ and $t_3$, and for $D_3(\mathbf{x})$ in terms of $t_2,t_3$ and the matrices $M_i$ above. The results are in Table \ref{tab:K3_S1}. Notice that every coefficient $\lambda_i$ lies on the interval $[0,\pi]$, and $\lambda_1+\lambda_2+\lambda_3 = \pi$. In any case, $D_3(\mathbf{x})$ is a convex combination of a subset of $\{M_1, M_2, M_3, M_4\}$.\\
	\indent Notice also that $\lambda_1$ and $\lambda_2$ determine the $d_{ij}$ in each row of Table \ref{tab:K3_S1}. When $\mathbf{x}$ is contained in a semicircle, both $\lambda_1$ and $\lambda_2$ are the distance between a pair of consecutive points. For example, when $t_3 \in [0,t_2]$, $x_3$ is between $x_1$ and $x_2$, and we have $d_{13} = \lambda_1$ and $d_{32} = \lambda_2$. Something similar happens when $t_3 \in [\pi, t_2+\pi]$: the point $-x_3$ is between $x_1$ and $x_2$, $d(x_1, -x_3) = \lambda_2$, and $d(-x_3, x_2) = \lambda_1$. Since $\lambda_1+\lambda_2+\lambda_3 = \pi$, we also have $d_{13} = \pi-\lambda_2$ and $d_{32} = \pi-\lambda_1$. Thus, not only can we write $D_3(\mathbf{x})$ in terms of $\lambda_1$ and $\lambda_2$, but $\lambda_1$ and $\lambda_2$ are the distance between a pair of consecutive points in some semicircle.
	
	\begin{table}[h]
	\begin{tabularx}{\textwidth}{lXXXl}
		Case & Distances & Coefficients & Relations & $D_3(\mathbf{x})$ equals \\
		\hline
		$t_3 \in [0, t_2]$
		&
		$d_{12} = t_2$ \newline
		$d_{13} = t_3$ \newline
		$d_{23} = t_2-t_3$
		&
		$\lambda_1 := t_3$ \newline
		$\lambda_2 := t_2-t_3$ \newline
		$\lambda_3 := \pi-t_2$
		&
		$d_{12} = \lambda_1+\lambda_2$ \newline
		$d_{13} = \lambda_1$ \newline
		$d_{23} = \lambda_2$
		&
		$\frac{\lambda_1}{\pi} M_1 + \frac{\lambda_2}{\pi} M_2 + \frac{\lambda_3}{\pi} M_4$
		\\
		\hline
		$t_3 \in [t_2, \pi]$
		&
		$d_{12} = t_2$ \newline
		$d_{13} = t_3$ \newline
		$d_{23} = t_3-t_2$
		&
		$\lambda_1 := t_2$ \newline
		$\lambda_2 := t_3-t_2$ \newline
		$\lambda_3 := \pi-t_3$
		&
		$d_{12} = \lambda_1$ \newline
		$d_{13} = \lambda_1+\lambda_2$ \newline
		$d_{23} = \lambda_2$
		&
		$\frac{\lambda_1}{\pi} M_1 + \frac{\lambda_2}{\pi} M_3 + \frac{\lambda_3}{\pi} M_4$
		\\
		\hline
		$t_3 \in [\pi, t_2+\pi]$
		&
		$d_{12} = t_2$ \newline
		$d_{13} = 2\pi-t_3$ \newline
		$d_{23} = t_3-t_2$
		&
		$\lambda_1 := t_2-t_3 \newline
		\phantom{\lambda_1:=} +\pi$ \newline
		$\lambda_2 := t_3-\pi$ \newline
		$\lambda_3 := \pi-t_2$
		&
		$d_{12} = \lambda_1+\lambda_2$ \newline
		$d_{13} = \lambda_1+\lambda_3$ \newline
		$d_{23} = \lambda_2+\lambda_3$
		&
		$\frac{\lambda_1}{\pi} M_1 + \frac{\lambda_2}{\pi} M_2 + \frac{\lambda_3}{\pi} M_3$.
		\\
		\hline
		$t_3 \in [t_2+\pi, 2\pi]$
		&
		$d_{12} = t_2$ \newline
		$d_{13} = 2\pi-t_3$ \newline
		$d_{23} = t_2-t_3 \newline
		\phantom{d_{23} = } +2\pi$
		&
		$\lambda_1 := t_2$ \newline
		$\lambda_2 := 2\pi-t_3$ \newline
		$\lambda_3 := t_3-t_2 \newline
		\phantom{\lambda_3 :=} -\pi$
		&
		$d_{12} = \lambda_2$ \newline
		$d_{13} = \lambda_1$ \newline
		$d_{23} = \lambda_1+\lambda_2$
		&
		$\frac{\lambda_1}{\pi} M_2 + \frac{\lambda_2}{\pi} M_3 + \frac{\lambda_3}{\pi} M_4$.
	\end{tabularx}
	\caption{The distance matrix $D_3(\mathbf{x})$ from Example \ref{ex:K3_S1} and its entries in terms of $t_2,t_3$ and $M_i$ for $1 \leq i \leq 4$.}
	\label{tab:K3_S1}
	\end{table}
\end{example}

\indent The objective of this section is to generalize Example \ref{ex:K3_S1} by showing that $\Kn_n(\Sp^1)$ is the geometric realization of a simplicial complex that we call the \emph{$n$-th state complex} $\St_n(\Sp^1)$. The State Complex will have features similar to those of Example \ref{ex:K3_S1}. Given $\mathbf{t} := (t_1, t_2, t_3) \in \{0\} \times [0,\pi] \times [0,2\pi]$, define $\mathbf{x}(\mathbf{t}) \in \T^3$ by $x_i(\mathbf{t}) := \exp(t_i \cdot \im)$. The function defined by $\mathbf{t} \mapsto D_3(\mathbf{x}(\mathbf{t}))$ is piecewise affine, and affine when restricted to each of the four cases in Table \ref{tab:K3_S1}. Moreover, the breakpoints occur when $x_3(\mathbf{t})$ is equal or antipodal to either $x_1(\mathbf{t})$ or $x_2(\mathbf{t})$. The $n$-th State Complex generalizes these features by defining a family of functions $\Phi_{A}: \Delta_{m} \to \T^n$ parametrized by the simplices $A \in \St_n(\Sp^1)$ such that the compositions $D_n \circ \Phi_{A}$ are affine. Moreover, if $A$ and $B$ are simplices of $\St_n(\Sp^1)$ with non-empty intersection, then $\Im(\Phi_{A \cap B})$ is the set of $\mathbf{x} \in \Im(\Phi_{A})$ or ($\rho(\mathbf{x}) \in \Im(\Phi_{A})$) such that a prescribed set of components of $\mathbf{x}$ are equal or antipodal to one another. A similar observation can be made for $\Im(\Phi_{A \cap B})$ as a subset of $\Im(\Phi_{B})$ and, in both cases, the set of components is determined by $A$ and $B$.

\begin{defn}
	\label{def:m_simplex}
	We model the $m$-simplex as
	\begin{equation*}
		\Delta_{m} := \{ \mathbf{t} \in \R^{m+1}: 0 \leq t_i \leq 1 \text{ and } t_1 + \dots + t_{m+1} = 1 \}.
	\end{equation*}
	Given $\mathbf{t} \in \Delta_{m}$, define $S_1(\mathbf{t}) := 0$ and $S_j(\mathbf{t}) := t_1 + \cdots t_{j-1}$ for $2 \leq j \leq m+2$.
\end{defn}

\subsection{Cluster structures}
\label{subsec:cluster_structures}
In Example \ref{ex:K3_S1}, we wrote every entry of $D_3(\mathbf{x})$ in terms of $\lambda_1$ and $\lambda_2$, which turned out to be the distances between pairs of consecutive points in a semicircle. We also showed that $D_3(\mathbf{x})$ is a convex combination of a subset of $M_1, \dots, M_4$. The objective of this section is to generalize these two observations. For example, when $\mathbf{x} \in \T^n$ is contained in a semicircle, there exists a permutation $\tau \in S_n$ and points $0 \leq y_1 \leq \cdots y_n \leq \pi$ such that $d(x_i, x_j) = |y_{\tau(i)}-y_{\tau(j)}|$. The distances $|y_{\tau(i)}-y_{\tau(j)}|$ and, hence, the matrix $D_n(\mathbf{x})$ are determined by $\tau$ and by the distances between consecutive points $y_i$. For a general point $\mathbf{x} \in \T^n$, we emulate the case when $t_3 \in [\pi, t_2+\pi]$ by replacing $\mathbf{x}$ with an $\widetilde{\mathbf{x}} \in \T^n$ whose components are contained in a semicircle in such a way that the distances between consecutive components of $\widetilde{\mathbf{x}}$ and their ``order'' determine $D_n(\mathbf{x})$.\\
\indent We start by choosing a map $\mathbf{x} \mapsto \widetilde{\mathbf{x}}$ and construct an indexing function $\mathfrak{c}$ that tracks the components of $\mathbf{x}$ that change in $\widetilde{\mathbf{x}}$ and the order of the components of $\widetilde{\mathbf{x}}$ in the semicircle that contains it. Then, for a fixed function $\mathfrak{c}$, we construct a function $\Phi_{\mathfrak{c}}:\Delta_{m-1} \to \T^{n}$ such that, for any $\mathbf{x} \in \Im(\Phi_{\mathfrak{c}})$, the order of the components of $\widetilde{\mathbf{x}}$ is determined by $\mathfrak{c}$. Fixing an order on $\widetilde{\mathbf{x}}$ allows us to write $D_n(\mathbf{x})$ in terms of the distances between the components of $\widetilde{\mathbf{x}}$ for any $\mathbf{x} \in \Im(\Phi_{\mathfrak{c}})$, and we show that this expression is a convex combination of a set of matrices determined by $\mathfrak{c}$.

\begin{defn}
	\label{def:x_tilde}
	
	Define the \emph{chirality function} $\sigma:\Sp^1 \to \{-1, +1\}$ by
	\begin{equation*}
		\sigma(x) :=
		\begin{cases}
			+1, & \text{ if } \arg(x) \in [0, \pi),\\
			-1, & \text{ if } \arg(x) \in [\pi, 2\pi).
		\end{cases}
	\end{equation*}
	Given $\mathbf{x} \in \T^n$, define $\widetilde{\mathbf{x}} := (\widetilde{x}_1, \dots, \widetilde{x}_n) \in \T^n$ where $\widetilde{x}_i := \sigma(x_i/x_1) \cdot x_i$. Observe that the components of $\widetilde{\mathbf{x}}$ are contained in the semicircle $\arg\inv(I)$, where $I = [\arg(x_1), \arg(x_1)+\pi) \mod(2\pi)$.
\end{defn}
\begin{remark}
	\label{rmk:chirality_function}
	Notice that $\sigma(\epsilon z) = \epsilon \sigma(z)$ when $\epsilon \in \{+1, -1\}$.
\end{remark}

\begin{defn}
	\label{def:cluster_structure_x}
	Let $\mathbf{x} \in \T^n$. Let $1 \leq m \leq n$ be the cardinality of $\{ \widetilde{x}_1, \dots, \widetilde{x}_n \}$. Set $k_1 := 1$, and choose indices $1 < k_j \leq n$ for $2 \leq j \leq m$ such that $\widetilde{x}_{k_1}, \dots, \widetilde{x}_{k_m}$ appear in anticlockwise order and are pairwise distinct. Define $\mathfrak{c}_{\mathbf{x}}:\{1, \dots, n\} \to \{ \pm 1, \dots, \pm m\}$ by $\mathfrak{c}_{\mathbf{x}}(i) := \sigma(x_j/x_1) \cdot j$ where $j$ satisfies $\widetilde{x}_i = \widetilde{x}_{k_j}$. We call $\mathfrak{c}_{\mathbf{x}}$ the \emph{cluster structure} induced by $\mathbf{x}$.
\end{defn}

In our later constructions, it will be useful to have an abstract notion of a cluster structure, i.e. a cluster structure that wasn't necessarily induced by a specific point $\mathbf{x} \in \T^n$.

\begin{defn}
	\label{def:cluster_structure}
	Let $1 \leq m \leq n$ be integers. An \emph{$(m,n)$-cluster structure} is a function $\mathfrak{c}:\{1, \dots, n\} \to \{\pm 1, \dots, \pm m\}$ such that $\mathfrak{c}(1) = +1$ and $|\mathfrak{c}|:\{1, \dots, n\} \to \{1, \dots, m\}$ is surjective. We say that $\mathfrak{c}$ has \emph{$m$ degrees of freedom}.
\end{defn}

\begin{remark}[$\mathfrak{c}_{\mathbf{x}}$ is an $(m,n)$-cluster structure.]
	$\mathfrak{c}_{\mathbf{x}}: \{1, \dots, n\} \to \{ \pm 1, \dots, \pm m\}$ satisfies Definition \ref{def:cluster_structure}. Indeed, if $k_1, \dots, k_m$ are indices as in Definition \ref{def:cluster_structure_x}, then $|\mathfrak{c}_{\mathbf{x}}(k_j)| = j$. Hence, $|\mathfrak{c}_{\mathbf{x}}|:\{1, \dots, n\} \to \{1, \dots, m\}$ is surjective. In particular, $k_1=1$ implies $\mathfrak{c}_{\mathbf{x}}(1) = \mathfrak{c}_{\mathbf{x}}(k_1) = \sigma(x_{k_1}/x_1) = 1$.
\end{remark}

The main use of an abstract cluster structure is to define a function $\Phi_{\mathfrak{c}}$ so that most points in the image of $\Phi_{\mathfrak{c}}$ have $\mathfrak{c}$ as their induced cluster structure.

\begin{defn}
	\label{def:phi_c}
	Let $1 \leq m \leq n$ be integers, and let $\mathfrak{c}$ be an $(m,n)$-cluster structure. Define
	\begin{align*}
		\Phi_{\mathfrak{c}}: \Delta_{m-1} &\to \{1\} \times \T^{n-1} \\
		\mathbf{t} & \mapsto (x_1, \dots, x_n),
	\end{align*}
	where $x_i = \sign(\mathfrak{c}(i)) \cdot \exp(S_{|\mathfrak{c}(i)|}(\mathbf{t}) \cdot \pi\im)$.
\end{defn}

\indent The idea behind $\Phi_{\mathfrak{c}}$ is to construct $\mathbf{x} = \Phi_{\mathfrak{c}}(\mathbf{t})$ by setting each $x_i$ to be equal or antipodal to one of the $m$ points
\begin{equation}
	\label{eq:Phi_c_example}
	\exp(S_1(\mathbf{t}) \cdot \pi \im), \dots, \exp(S_m(\mathbf{t}) \cdot \pi \im) \in \Sp^1.
\end{equation}
The choice is determined by $\mathfrak{c}$: $x_i$ is equal to the $j$-th point $\exp(S_j(\mathbf{t}) \cdot \pi \im)$ if $\mathfrak{c}(i)=j$ or antipodal to it if $\mathfrak{c}(i)=-j$. Moreover, if $\mathbf{t}$ is a point in the interior of $\Delta_{m-1}$ (i.e. $t_j>0$ for all $1 \leq j \leq m$), then the points in Equation (\ref{eq:Phi_c_example}) are all distinct and appear in anticlockwise order from left to right. We can also check that $\sigma(x_i/x_1) = \sign(\mathfrak{c}(i))$, so $\widetilde{x}_i = \exp(S_{|\mathfrak{c}(i)|}(\mathbf{t}) \cdot \pi \im)$. These facts imply that the cluster structure induced by $\mathbf{x}$ is $\mathfrak{c}$. We later calculate the cluster structure induced by an arbitrary point in the image of $\Phi_{\mathfrak{c}}$ in Lemma \ref{lemma:cluster_boundary}.\\
\indent Conversely, if $\mathfrak{c}$ is the cluster structure induced by $\mathbf{x}$, then $\mathbf{x}$ must be in the image of $\Phi_{\mathfrak{c}}$.

\begin{prop}
	\label{prop:cluster_structure_x}
	Let $\mathbf{x} \in \{1\} \times \T^{n-1}$ and let $\mathfrak{c} := \mathfrak{c}_{\mathbf{x}}$ be the $(m,n)$-cluster structure induced by $\mathbf{x}$. Choose indices $k_j$ for $1 \leq j \leq m$ as in Definition \ref{def:cluster_structure_x}. Let $t_j := d(\widetilde{x}_{k_{j}}, \widetilde{x}_{k_{j+1}})/\pi$ for $1 \leq j \leq m-1$ and $t_m := d(\widetilde{x}_{k_m}, e^{\pi \im})/\pi$. Then $\mathbf{t}:= (t_1, \dots, t_{m}) \in \Delta_{m-1}$ and $\mathbf{x} = \Phi_{\mathfrak{c}}(\mathbf{t})$.
\end{prop}
\begin{proof}
	By the definitions of $\widetilde{\mathbf{x}}$ and $\mathfrak{c}_{\mathbf{x}}$, the points $\widetilde{x}_{k_1}, \dots, \widetilde{x}_{k_m}$ appear in anticlockwise order in a semicircle from $\widetilde{x}_{k_1} = +1$ to $e^{\pi \im} = -1$. Then $S_{j}(\mathbf{t}) = \big(d(\widetilde{x}_{k_1},\widetilde{x}_{k_2}) + \cdots + d(\widetilde{x}_{k_{j-1}},\widetilde{x}_{k_j}) \big)/\pi = d(\widetilde{x}_{k_1},\widetilde{x}_{k_j})/\pi$ for $1 \leq j \leq m$. In particular, $S_{m}(\mathbf{t}) + t_{m} = \big(d(\widetilde{x}_{k_1}, \widetilde{x}_{k_m}) + d(\widetilde{x}_{k_m}, e^{\pi \im})\big)/\pi = 1$, so $\mathbf{t} \in \Delta_{m-1}$. These equations also imply $\widetilde{x}_{k_j} = \exp(S_j(\mathbf{t}) \cdot \pi\im)$. On the other hand, for any $1 \leq i \leq n$, $x_i = \sigma(x_i/x_1) \cdot \widetilde{x}_{i} = \sigma(x_i/x_1) \cdot \widetilde{x}_{k_j}$ for some $1 \leq j \leq m$ by construction of the $k_j$. Then $\mathfrak{c}(i) = \sigma(x_i/x_1) \cdot j$, so the $i$-th component of $\Phi_{\mathfrak{c}}(\mathbf{t})$ is $\sign(\mathfrak{c}(i)) \cdot \exp(S_{|\mathfrak{c}(i)|}(\mathbf{t}) \cdot \pi \im) = \sigma(x_i/x_1) \cdot \exp(S_{j}(\mathbf{t}) \cdot \pi \im) = \sigma(x_i/x_1) \cdot \widetilde{x}_{k_j} = x_i$. In conclusion, $\mathbf{x} = \Phi_{\mathfrak{c}}(\mathbf{t})$.
\end{proof}

The next proposition shows that $\mathfrak{c}$ determines the shape of $D_n(\mathbf{x})$.

\begin{prop}
	\label{prop:distances_Phi_c}
	Let $\mathfrak{c}$ be an $(m,n)$-cluster structure. Given $\mathbf{t} \in \Delta_{m-1}$, let $\mathbf{x} = \Phi_{\mathfrak{c}}(\mathbf{t})$ and $d_{ij} := d(x_i, x_j)$. Then
	\begin{equation}
		\label{eq:distances_Phi_c}
		d_{ij} =
		\begin{cases}
			|S_{|\mathfrak{c}(j)|}(\mathbf{t}) - S_{|\mathfrak{c}(i)|}(\mathbf{t})| \cdot \pi & \text{if } \sign(\mathfrak{c}(i)) = \sign(\mathfrak{c}(j)) \\
			(1 - |S_{|\mathfrak{c}(j)|}(\mathbf{t}) - S_{|\mathfrak{c}(i)|}(\mathbf{t})|) \cdot \pi & \text{if } \sign(\mathfrak{c}(i)) \neq \sign(\mathfrak{c}(j)).
		\end{cases}
	\end{equation}
\end{prop}
\begin{remark}
	Notice that, if $i \leq j$, then $|S_j(\mathbf{t}) - S_i(\mathbf{t})| = t_i + \dots + t_{j-1}$. Additionally, the two cases in Equation (\ref{eq:distances_Phi_c}) depend on $\mathfrak{c}$, not on $\mathbf{t}$, so $d_{ij}$ is an affine function of $\mathbf{t}$ regardless of the values of $\mathfrak{c}(i)$ and $\mathfrak{c}(j)$.
\end{remark}
\begin{proof}
	Let $a := |\mathfrak{c}(i)|$ and $b := |\mathfrak{c}(j)|$ and assume $a \leq b$ without loss of generality. If $a > b$, the proposition follows from the symmetry of $d$. Recall that $d_{ij} = \min \left( \arg(\frac{x_j}{x_i}), \arg(\frac{x_i}{x_j}) \right)$. If $\sign(\mathfrak{c}(i)) = \sign(\mathfrak{c}(j))$, then $\arg(\frac{x_j}{x_i}) = \arg\big(\exp( [S_{b}(\mathbf{t}) - S_{a}(\mathbf{t})] \cdot \pi \im) \big) = [S_{b}(\mathbf{t}) - S_{a}(\mathbf{t})] \cdot \pi$ and $\arg(\frac{x_i}{x_j}) = 2\pi - [S_{b}(\mathbf{t}) - S_{a}(\mathbf{t})] \cdot \pi$. Notice that $S_{b}(\mathbf{t}) \leq 1$, so $S_{b}(\mathbf{t}) - S_{a}(\mathbf{t}) \leq 1 \leq 2 - [S_{b}(\mathbf{t}) - S_{a}(\mathbf{t})]$. Thus, $d_{ij} = [S_{b}(\mathbf{t}) - S_{a}(\mathbf{t})] \cdot \pi$. If $\sign(\mathfrak{c}(i)) \neq \sign(\mathfrak{c}(j))$ instead, we have $\arg(\frac{x_j}{x_i}) = \arg\big(-\exp([S_{b}(\mathbf{t}) - S_{a}(\mathbf{t})] \cdot \pi \im) \big) = (1 + [S_{b}(\mathbf{t}) - S_{a}(\mathbf{t}) ]) \cdot \pi$ and $\arg(\frac{x_i}{x_j}) = (1 - [S_{b}(\mathbf{t}) - S_{a}(\mathbf{t})] ) \cdot \pi$. Hence, $d_{ij} = (1 - [S_{b}(\mathbf{t}) - S_{a}(\mathbf{t})] ) \cdot \pi$.
\end{proof}

\indent The previous two propositions achieve our objective of writing $D_n(\mathbf{x})$ in terms of the distances between points of $\widetilde{\mathbf{x}}$. Proposition \ref{prop:cluster_structure_x} shows that every $\mathbf{x} \in \{1\} \times \T^{n-1}$ can be written as $\Phi_{\mathfrak{c}}(\mathbf{t})$, where $\mathfrak{c} = \mathfrak{c}_{\mathbf{x}}$ and $\mathbf{t} \in \Delta_{m-1}$ is defined in terms of the distances between the components of $\widetilde{\mathbf{x}}$, while Proposition \ref{prop:distances_Phi_c} gives $D_n(\mathbf{x})$ in terms of $\mathbf{t}$. Additionally, we use this expression to verify that $\Phi_{\mathfrak{c}}$ is injective.

\begin{corollary}
	\label{cor:Phi_c_injective}
	For any $(m,n)$-cluster structure $\mathfrak{c}$, $\Phi_{\mathfrak{c}}$ and $D_n \circ \Phi_{\mathfrak{c}}$ are injective.
\end{corollary}
\begin{proof}
	Let $\mathbf{t}, \mathbf{s} \in \Delta_{m-1}$ such that $t_{r} \neq s_{r}$ for some $1 \leq r \leq m$. Let $\mathbf{x} := \Phi_{\mathfrak{c}}(\mathbf{t})$ and $\mathbf{y} := \Phi_{\mathfrak{c}}(\mathbf{s})$. Since $|\mathfrak{c}|:\{1, \dots, n\} \to \{1, \dots, m\}$ is surjective, we can choose $1 \leq k_1, \cdots, k_m \leq n$ such that $|\mathfrak{c}(k_j)| = j$. If $r < m$, then $|S_{r+1}(\mathbf{t}) - S_{r}(\mathbf{t})| = t_{r}$. By Proposition \ref{prop:distances_Phi_c}, either $d(x_{k_r},x_{k_{r+1}}) = t_r \pi$ and $d(y_{k_r},y_{k_{r+1}}) = s_r \pi$, or $d(x_{k_r},x_{k_{r+1}}) = (1-t_r) \pi$ and $d(y_{k_r},y_{k_{r+1}}) = (1-s_r) \pi$. In both cases, $d(x_{k_r},x_{k_{r+1}}) \neq d(y_{k_r},y_{k_{r+1}})$. If $r=m$, then $|S_1(\mathbf{t})-S_m(\mathbf{t})| = t_1 + \cdots +t_{m-1} = 1-t_m$, and we get $d(x_{k_r},x_{k_{1}}) \neq d(y_{k_r},y_{k_{1}})$. In any case, $D_n \circ \Phi_{\mathfrak{c}}(\mathbf{t}) = D_n(\mathbf{x}) \neq D_n(\mathbf{y}) = D_n \circ \Phi_{\mathfrak{c}}(\mathbf{s})$, and thus, $\Phi_{\mathfrak{c}}(\mathbf{t}) \neq \Phi_{\mathfrak{c}}(\mathbf{s})$.
\end{proof}

Let's now show that $D_n(\mathbf{x})$ is a convex combination of a set of matrices determined by $\mathfrak{c}$. These matrices are the image under $D_n$ of any of the points in the following definition.
\begin{defn}
	\label{def:cluster_structure_vertices}
	Let $\mathfrak{c}$ be an $(m,n)$-cluster structure. For $1 \leq k \leq m$, define $\mathbf{t}^{(k)} \in \Delta_{m-1}$ by $t^{(k)}_k := 1$ and $t^{(k)}_i := 0$ for $i \neq k$. Define $\mathbf{v}^{(k)}(\mathfrak{c}) := \Phi_{\mathfrak{c}}(\mathbf{t}^{(k)})$ and $V(\mathfrak{c}) := \{ \mathbf{v}^{(k)}(\mathfrak{c}) \mid 1 \leq k \leq m \}$.
\end{defn}

\begin{remark}
	\label{rmk:vertices_of_c}
	Let $\mathfrak{c}$ be an $(m,n)$-cluster structure, and $\mathbf{v}^{(k)} := \mathbf{v}^{(k)}(\mathfrak{c})$. Then
	\begin{equation*}
		v^{(k)}_i =
		\begin{cases}
			\sign(\mathfrak{c}(i)),		& |\mathfrak{c}(i)| \leq k\\
			-\sign(\mathfrak{c}(i)),	& |\mathfrak{c}(i)| > k.
		\end{cases}
	\end{equation*}
	This follows from $v^{(k)}_i = \sign(\mathfrak{c}(i)) \exp(S_{|\mathfrak{c}(i)|}(\mathbf{t}^{(k)}) \cdot \pi \im)$, which is the expression for the $i$-th component of $\Phi_{\mathfrak{c}}$, and the fact that $S_{j}(\mathbf{t}^{(k)}) = 0$ if $j \leq k$ and $S_{j}(\mathbf{t}^{(k)}) = 1$ when $j > k$.
\end{remark}

\begin{prop}
	\label{prop:cluster_convex_hull}
	Let $\mathfrak{c}$ be an $(m,n)$-cluster structure, and let $\mathbf{t} \in \Delta_{m-1}$. Then 
	\begin{equation*}
		D_n(\Phi_{\mathfrak{c}}(\mathbf{t})) = t_1 D_n(\mathbf{v}^{(1)}(\mathfrak{c})) + \cdots + t_m D_n(\mathbf{v}^{(m)}(\mathfrak{c})),
	\end{equation*}
	and, as a consequence, $\Im(D_n \circ \Phi_{\mathfrak{c}}) = \conv( D_n \circ V(\mathfrak{c}) )$.
\end{prop}
\begin{proof}
	Let $\mathbf{t} \in \Delta_{m-1}$ be arbitrary and let $\mathbf{x} := \Phi_{\mathfrak{c}}(\mathbf{t})$. Let $\mathbf{v}^{(k)} := \mathbf{v}^{(k)}(\mathfrak{c})$, and denote $d_{ij} := d(x_i, x_j)$ and $d_{ij}^{(k)} := d\big( v^{(k)}_i, v^{(k)}_j \big)$. We claim that $d_{ij} = t_1 d_{ij}^{(1)} + \cdots + t_m d_{ij}^{(m)}$.\\
	\indent Without loss of generality, suppose that $|\mathfrak{c}(i)| \leq |\mathfrak{c}(j)|$. Otherwise, the claim follows by symmetry of $d$. We have two cases depending on whether $\sign(\mathfrak{c}(i))$ and $\sign(\mathfrak{c}(j))$ are equal. If $\sign(\mathfrak{c}(i)) = \sign(\mathfrak{c}(j))$, Remark \ref{rmk:vertices_of_c} gives $v^{(k)}_i = -\sign(\mathfrak{c}(i)) = -\sign(\mathfrak{c}(j)) = v^{(k)}_j$ when $1 \leq k < |\mathfrak{c}(i)|$. Similarly, if $|\mathfrak{c}(j)| \leq k \leq m$, we have $v^{(k)}_i = \sign(\mathfrak{c}(i)) = \sign(\mathfrak{c}(j)) = v^{(k)}_j$. However, when $|\mathfrak{c}(i)| \leq k < |\mathfrak{c}(j)|$, $v^{(k)}_i = \sign(\mathfrak{c}(i)) = \sign(\mathfrak{c}(j)) = -v^{(k)}_j$. Together, these cases give $d_{ij}^{(k)} = \pi$ if $|\mathfrak{c}(i)| \leq k < |\mathfrak{c}(j)|$ and $d_{ij}^{(k)} = 0$ otherwise. Then
	\begin{align*}
		t_1 d_{ij}^{(1)} + \cdots + t_m d_{ij}^{(m)}
		&= \sum_{|\mathfrak{c}(i)| \leq k < |\mathfrak{c}(j)|} t_k \cdot \pi = |S_{|\mathfrak{c}(j)|}(\mathbf{t}) - S_{|\mathfrak{c}(i)|}(\mathbf{t})| \cdot \pi = d_{ij}
	\end{align*}
	by Proposition \ref{prop:distances_Phi_c}. If $\sign(\mathfrak{c}(i)) \neq \sign(\mathfrak{c}(j))$, the values of $d_{ij}^{(k)}$ are reversed and, similarly to the previous paragraph, we can show that $d_{ij}^{(k)} = 0$ if $|\mathfrak{c}(i)| \leq k < |\mathfrak{c}(j)|$ and $d_{ij}^{(k)} = \pi$ otherwise. Since $t_1 + \cdots + t_m = 1$, we get
	\begin{align*}
		t_1 d_{ij}^{(1)} + \cdots + t_m d_{ij}^{(m)}
		&= \sum_{1 \leq k < |\mathfrak{c}(i)|} t_k \cdot \pi + \sum_{|\mathfrak{c}(j)| \leq k \leq m} t_k \cdot \pi \\
		&= \pi - \sum_{|\mathfrak{c}(i)| \leq k < |\mathfrak{c}(j)|} t_k \cdot \pi
		= \left(1 - |S_{|\mathfrak{c}(j)|}(\mathbf{t}) - S_{|\mathfrak{c}(i)|}(\mathbf{t})| \right) \cdot \pi
		= d_{ij}.
	\end{align*}
\end{proof}

\subsection{The definition of $\St_n(\Sp^1)$}
\label{subsec:state_complex}
Let's turn now to the definition of the State Complex $\St_n(\Sp^1)$. Our definition of $V(\mathfrak{c})$ might be suggestive: the $n$-th State Complex is the collection of vertex sets of all possible cluster structures.

\begin{defn}
	\label{def:state_complex}
	The \emph{$n$-th State Complex} is the collection
	\begin{equation*}
		\St_n(\Sp^1) := \{ V(\mathfrak{c}) \mid \mathfrak{c} \text{ is an $(m,n)$-cluster structure with } 1 \leq m \leq n \}.
	\end{equation*}
\end{defn}

\begin{remark}
	\label{rmk:vertices_of_St_n}
	In Remark \ref{rmk:vertices_of_c}, we saw that every element of $V(\mathfrak{c})$ is a point $\mathbf{v} \in \{1\} \times \T^{n-1}$ such that $v_i \in \{+1, -1\}$ for all $2 \leq i \leq n$. Now we prove the converse: every such point is the vertex of some cluster structure $\mathfrak{c}$. Our candidate is the $(1,n)$-cluster structure $\mathfrak{c}: \{1, \dots, n\} \to \{ \pm 1\}$ defined by $\mathfrak{c}(i) = v_i$. Since $\Delta_0 = \{1\} \subset \R$, $V(\mathfrak{c})$ is a singleton. In fact, for $\mathbf{t} \in \Delta_0$, the $i$-th component of $\Phi_{\mathfrak{c}}(\mathbf{t})$ is $\sign(\mathfrak{c}(i)) \exp(S_{|\mathfrak{c}(i)|}(\mathbf{t}) \cdot \pi \im) = \sign(\mathfrak{c}(i)) \exp(S_{1}(\mathbf{t}) \cdot \pi \im) = \sign(\mathfrak{c}(i)) = v_i$. Hence, $V(\mathfrak{c}) = \Im(\Phi_{\mathfrak{c}}) = \{\mathbf{v}\}$.
\end{remark}

To prove that $\St_n(\Sp^1)$ is an abstract simplicial complex, we must find, for every $A \subset V(\mathfrak{c})$, a cluster structure $\mathfrak{c}_I$  such that $V(\mathfrak{c}_I) = A$. The next definition gives such a $\mathfrak{c}_I$.

\begin{defn}
	\label{def:Phi_c_restricted}
	Let $\mathfrak{c}$ be an $(m,n)$-cluster structure. Let $1 \leq \ell \leq m$, define $k_0 := 0$, and choose $1 \leq k_1 < \cdots < k_\ell \leq m$. Set $I := \{k_1, \dots, k_\ell\}$. Define the $(\ell,n)$-cluster structure $\mathfrak{c}_I$ by
	\begin{equation*}
		\mathfrak{c}_I(i) :=
		\begin{cases}
			\sign(\mathfrak{c}(i)) \cdot j,	& k_{j-1} < |\mathfrak{c}(i)| \leq k_{j} \text{ for some } 1 \leq j \leq \ell,\\
			-\sign(\mathfrak{c}(i)),		& k_{\ell \phantom{-1}} < |\mathfrak{c}(i)|.
		\end{cases}
	\end{equation*}
\end{defn}
Notice that $\mathfrak{c}_I$ is indeed an $(\ell, n)$-cluster structure because $|\mathfrak{c}_I(k_j)|=j$ for all $1 \leq j \leq \ell$, and $k_1=1$ implies $\mathfrak{c}_I(1) = \sign(\mathfrak{c}(1)) \cdot 1 = 1$. Moreover, our definition of $\mathfrak{c}_I$ is set up so that $\Phi_{\mathfrak{c}_I}$ is a restriction of $\Phi_{\mathfrak{c}}$.

\begin{lemma}
	\label{lemma:Phi_c_restricted}
	Let $\mathfrak{c}$ be an $(m,n)$-cluster structure. Let $1 \leq \ell \leq m$, choose $1 \leq k_1 < \cdots < k_\ell \leq m$, and set $I := \{k_1, \dots, k_\ell\}$. Let $\mathbf{t} \in \Delta_{m-1}$ such that $t_i = 0$ if $i \notin I$, and define $\mathbf{s} := (t_{k_1}, \dots, t_{k_\ell}) \in \Delta_{\ell-1}$. Then $\Phi_{\mathfrak{c}_I}(\mathbf{s}) = \Phi_{\mathfrak{c}}(\mathbf{t})$.
\end{lemma}
\begin{proof}
	Let $k_0 := 0$ and $1 \leq i \leq n$. We need two cases to verify that the $i$-th components of $\Phi_{\mathfrak{c}}(\mathbf{t})$ and $\Phi_{\mathfrak{c}_I}(\mathbf{s})$ agree. For the first case, suppose that $k_{j-1} < |\mathfrak{c}(i)| \leq k_j$ for some $1 \leq j \leq \ell$. Then $|\mathfrak{c}_I(i)| = j$, and
	\begin{equation*}
		S_{|\mathfrak{c}(i)|}(\mathbf{t}) = t_{k_1} + \cdots + t_{k_{j-1}} = S_{|\mathfrak{c}_I(i)|}(\mathbf{s}).
	\end{equation*}
	 Moreover, $\sign(\mathfrak{c}_I(i)) = \sign(\mathfrak{c}(i))$, so $\sign(\mathfrak{c}(i)) \exp( S_{|\mathfrak{c}(i)|}(\mathbf{t}) \cdot \pi \im) = \sign(\mathfrak{c}_I(i)) \exp( S_{|\mathfrak{c}_I(i)|}(\mathbf{s}) \cdot \pi \im)$. For the second case, suppose $|\mathfrak{c}(i)| > k_\ell$. Now we have $S_{|\mathfrak{c}(i)|}(\mathbf{t}) = t_{k_1} + \cdots + t_{k_{\ell}} = 1$, and since $|\mathfrak{c}_I(i)|=1$, we get
	\begin{equation*}
		\sign(\mathfrak{c}(i)) \exp( S_{|\mathfrak{c}(i)|}(\mathbf{t}) \cdot \pi \im) = - \sign(\mathfrak{c}(i)) = \sign(\mathfrak{c}_I(i)) = \sign(\mathfrak{c}_I(i)) \exp( S_{|\mathfrak{c}_I(i)|}(\mathbf{t}) \cdot \pi \im).
	\end{equation*}
	In either case, the $i$-th components of $\Phi_{\mathfrak{c}}(\mathbf{t})$ and $\Phi_{\mathfrak{c}_I}(\mathbf{s})$ agree.
\end{proof}

Conversely, for any $(m,n)$-cluster structure $\mathfrak{c}$ and any $\mathbf{x} \in \Im(\Phi_{\mathfrak{c}})$, $\mathfrak{c}_{\mathbf{x}}$ must be of the form $\mathfrak{c}_I$ for some $I \subset \{1, \dots, m\}$.
\begin{lemma}
	\label{lemma:cluster_boundary}
	Let $\mathfrak{c}$ be an $(m,n)$-cluster structure. Let $1 \leq \ell \leq m$, choose $1 \leq k_1 < \cdots < k_\ell \leq m$, and set $I := \{k_1, \dots, k_\ell\}$. Let $\mathbf{t} \in \Delta_{m-1}$ such that $t_i > 0$ if and only if $i \in I$, and define $\mathbf{x} = \Phi_{\mathfrak{c}}(\mathbf{t})$. Then $\mathfrak{c}_{\mathbf{x}} = \mathfrak{c}_I$.
\end{lemma}
\begin{proof}
	Let $k_0 := 0$. By hypothesis,
	\begin{equation}
		\label{eq:cluster_boundary}
		0 = S_1(\mathbf{t}) = \cdots = S_{k_1}(\mathbf{t})
		< S_{k_1+1}(\mathbf{t}) = \cdots = S_{k_2}(\mathbf{t}) < \cdots < S_{k_\ell+1}(\mathbf{t}) = \cdots = S_{m+1}(\mathbf{t}) = 1.
	\end{equation}
	Recall that $x_i = \sign(\mathfrak{c}(i)) \exp(S_{|\mathfrak{c}(i)|}(\mathbf{t}) \cdot \pi \im)$ by definition of $\Phi_{\mathfrak{c}}$. If $k_{j-1} < |\mathfrak{c}(i)| \leq k_{j}$ for some $1 \leq j \leq \ell$, then $S_{|\mathfrak{c}(i)|}(\mathbf{t})$ is strictly less than $S_{k_\ell+1}(\mathbf{t}) = 1$. Hence, $\arg\big(\exp(S_{|\mathfrak{c}(i)|}(\mathbf{t}) \cdot \pi \im) \big)$ lies on the interval $[0, \pi)$, so $\sigma(x_i/x_1) = \sigma(x_i) = \sign(\mathfrak{c}(i))$ by Remark \ref{rmk:chirality_function}. Together with the equality $S_{|\mathfrak{c}(i)|}(\mathbf{t}) = S_{k_{j}}(\mathbf{t})$, we get $\widetilde{x}_i = \sigma(x_i/x_1) x_i = \exp(S_{|\mathfrak{c}(i)|}(\mathbf{t}) \cdot \pi \im) = \exp(S_{k_{j}}(\mathbf{t}) \cdot \pi \im)$. This reasoning applies to $k_{j}$ in particular, so $\widetilde{x}_i = \widetilde{x}_{k_{j}}$. If $k_\ell < |\mathfrak{c}(i)|$, then $x_{i} = \sign(\mathfrak{c}(i)) \exp(S_{k_\ell}(\mathbf{t}) \cdot \pi \im) = \sign(\mathfrak{c}(i)) \exp(\pi \im) = - \sign(\mathfrak{c}(i))$. This yields $\widetilde{x}_i = +1 = \widetilde{x}_1$.\\
	\indent Together with Equation (\ref{eq:cluster_boundary}), the expressions $\widetilde{x}_{k_{j}} = \exp(S_{k_{j}}(\mathbf{t}) \cdot \pi \im)$ from the previous paragraph imply that the points $\widetilde{x}_{k_1}, \widetilde{x}_{k_2}, \cdots, \widetilde{x}_{k_{\ell}}$ appear in anticlockwise order. Moreover, $\widetilde{x}_i$ equals $\widetilde{x}_{k_{j}}$ when $k_{j-1} < |\mathfrak{c}(i)| \leq k_{j}$ and $1 \leq j \leq \ell$, and $\widetilde{x}_{k_1}$ when $k_\ell < |\mathfrak{c}(i)|$. In short, $\widetilde{x}_i = \widetilde{x}_{k_j}$ for some $1 \leq j \leq \ell$, so the indices $k_1, \dots, k_{\ell}$ satisfy Definition \ref{def:cluster_structure_x}. For this reason,
	\begin{itemize}
		\item $|\mathfrak{c}_{\mathbf{x}}(i)| = 1$ if $1 \leq |\mathfrak{c}(i)| \leq k_1$,
		\item $|\mathfrak{c}_{\mathbf{x}}(i)| = j$ if $k_{j-1} < |\mathfrak{c}(i)| \leq k_{j}$, and
		\item $|\mathfrak{c}_{\mathbf{x}}(i)| = 1$ if $k_\ell < |\mathfrak{c}(i)|$.
	\end{itemize}
	This is the desired expression for $\mathfrak{c}_{\mathbf{x}}$ up to sign. If $k_{j-1} < |\mathfrak{c}(i)| \leq k_{j}$ and $1 \leq j \leq \ell$, then $\sign(\mathfrak{c}_{\mathbf{x}}(i))$, $\sign(\mathfrak{c}(i))$, and $\sigma(x_i/x_1)$ are all equal: $\sign(\mathfrak{c}_{\mathbf{x}}(i)) = \sigma(x_i/x_1)$ holds by definition of $\mathfrak{c}_{\mathbf{x}}$, while $\sign(\mathfrak{c}(i)) = \sigma(x_i/x_1)$ follows from the reasoning in the first paragraph. In contrast, when $k_\ell < |\mathfrak{c}(i)|$, we have $x_{i} = - \sign(\mathfrak{c}(i))$, so $\sign(\mathfrak{c}_{\mathbf{x}}(i)) = \sigma(x_i/x_1) = - \sign(\mathfrak{c}(i))$. Thus, $\sign(\mathfrak{c}_{\mathbf{x}}(i)) = \sign(\mathfrak{c}(i))$ if and only if $1 \leq |\mathfrak{c}(i)| \leq k_\ell$.
\end{proof}

Thanks to Lemma \ref{lemma:Phi_c_restricted}, we can prove that every non-empty subset of $V(\mathfrak{c})$ is the vertex set of a cluster structure of the form $\mathfrak{c}_I$.
\begin{corollary}
	\label{cor:cluster_boundary}
	Let $\mathfrak{c}$ be an $(m,n)$-cluster structure. Let $1 \leq k_1 < \cdots < k_\ell \leq m$, and set $I := \{k_1, \dots, k_\ell\}$. Then $\mathbf{v}^{(i)}(\mathfrak{c}_{I}) = \mathbf{v}^{(k_i)}(\mathfrak{c})$ for $1 \leq i \leq \ell$ and $V(\mathfrak{c}_{I}) = \{ \mathbf{v}^{(k_1)}(\mathfrak{c}), \dots, \mathbf{v}^{(k_\ell)}(\mathfrak{c}) \} \subset V(\mathfrak{c})$. As a consequence, $\St_n(\Sp^1)$ is an abstract simplicial complex.
\end{corollary}
\begin{proof}
	Let $\mathbf{t}^{(k_i)} \in \Delta_{m-1}$ and $\mathbf{s}^{(i)} \in \Delta_{\ell-1}$ be such that $t^{(k_i)}_{k_i} = 1$, $t^{(k_i)}_{j} = 0$ for $j \neq k_i$, $s^{(i)}_{i} = 1$, and $s^{(i)}_{j} = 0$ for $j \neq i$. Notice that $\mathbf{s}^{(i)} = (t^{(k_i)}_{k_1}, \dots, t^{(k_i)}_{k_\ell})$, so Lemma \ref{lemma:Phi_c_restricted} gives $\mathbf{v}^{(i)}(\mathfrak{c}_I) = \Phi_{\mathfrak{c}_I}(\mathbf{s}^{(i)}) = \Phi_{\mathfrak{c}}(\mathbf{t}^{(k_i)}) = \mathbf{v}^{(k_i)}(\mathfrak{c})$.
\end{proof}

A natural question at this point is whether cluster structures are determined by their vertices. As we saw in Remark \ref{rmk:vertices_of_St_n}, this is true for cluster structures with one degree of liberty. However, for every $(m,n)$-cluster structure $\mathfrak{c}$ with $m \geq 2$, the following definition gives a different $(m,n)$-cluster structure with vertex set $V(\mathfrak{c})$.

\begin{defn}
	\label{def:cluster_structure_transpose}
	Let $\mathfrak{c}$ be a $(m,n)$ cluster structure. Define the $(m,n)$-cluster structure $\rho \cdot \mathfrak{c}:\{1, \dots, n\} \to \{\pm 1, \dots, \pm m\}$ by
	\begin{equation}
		\label{eq:c_transpose}
		\rho \cdot \mathfrak{c}(i) :=
		\begin{cases}
			-\sign(\mathfrak{c}(i)) \left( m+2-|\mathfrak{c}(i)| \right), & \text{if } |\mathfrak{c}(i)| \neq 1,\\
			\mathfrak{c}(i), & \text{if } |\mathfrak{c}(i)| = 1.
		\end{cases}
	\end{equation}
\end{defn}

\indent We chose the notation $\rho \cdot \mathfrak{c}$ because, as we will soon see, $\rho \cdot \mathfrak{c}_{\mathbf{x}} = \mathfrak{c}_{\rho(\mathbf{x})}$. This also grants intuition for the formula in Equation (\ref{eq:c_transpose}). In the same way that the reflection $\rho$ exchanges the anticlockwise and clockwise directions, the function $c \mapsto m+2-c$ in the first case of Equation (\ref{eq:c_transpose}) sends the sequence $2, \dots, m$ to $m, \dots, 2$. As for the second case, $\rho \cdot \mathfrak{c}(i)$ equals $\mathfrak{c}(i)$ when the latter is $\pm 1$ and, in that case, $\Phi_{\mathfrak{c}}(\mathbf{t}) = \pm 1$ is a fixed point of $\rho$.\\
\indent We need some preparation to verify our claims about $\rho \cdot \mathfrak{c}$. The first step is defining an involution $\rho_\Delta: \Delta_{m-1} \to \Delta_{m-1}$ by
\begin{equation*}
	\rho_\Delta(t_1, \dots, t_{m}) := (t_{m}, \dots, t_1).
\end{equation*}

\begin{prop}
	\label{prop:c_transpose}
	Let $\mathfrak{c}$ be an $(m,n)$-cluster structure. Then $\rho \circ \Phi_\mathfrak{c} = \Phi_{\rho \cdot \mathfrak{c}} \circ \rho_\Delta$.
\end{prop}
\begin{proof}
	Let $\mathbf{t} \in \Delta_{m-1}$. If $|\mathfrak{c}(j)|=1$,
	\begin{align*}
		\big( \Phi_{\rho \cdot \mathfrak{c}} \circ \rho_\Delta(\mathbf{t}) \big)_j
		&= \sign(\rho \cdot \mathfrak{c}(j)) \exp\left(S_{|\rho \cdot \mathfrak{c}(j)|}(\rho_\Delta(\mathbf{t})) \cdot \pi \im \right) \\
		& = \sign(\mathfrak{c}(j)) \exp\left(S_{1}(\rho_\Delta(\mathbf{t})) \cdot \pi \im \right) \\
		&= \sign(\mathfrak{c}(j)) = \rho(\sign(\mathfrak{c}(j)))
		= \big( \rho \circ \Phi_{\mathfrak{c}} (\mathbf{t}) \big)_j.
	\end{align*}
	Notice that $S_{m+2-k}(\rho_\Delta(\mathbf{t})) = t_{m} + \cdots + t_{k} =  1-S_{k}(\mathbf{t})$ for $1 \leq k \leq m+1$. Then, if $|\mathfrak{c}(j)| \neq 1$, we get
	\begin{align*}
		\big( \Phi_{\rho \cdot \mathfrak{c}} \circ \rho_\Delta(\mathbf{t}) \big)_j
		&= \sign(\rho \cdot \mathfrak{c}(j)) \cdot \exp\left(S_{|\rho \cdot \mathfrak{c}(j)|}(\rho_\Delta(\mathbf{t})) \cdot \pi \im \right) \\
		&= -\sign(\mathfrak{c}(j)) \cdot \exp\left( S_{m+2-|\mathfrak{c}(j)|}(\rho_\Delta(\mathbf{t})) \cdot \pi \im \right) \\
		&= -\sign(\mathfrak{c}(j)) \cdot \exp\left( [1-S_{|\mathfrak{c}(j)|}(\mathbf{t})] \cdot \pi \im \right) \\
		&= +\sign(\mathfrak{c}(j)) \cdot \exp\left( -S_{|\mathfrak{c}(j)|}(\mathbf{t}) \cdot \pi \im \right) \\
		&= \rho\left[ \sign(\mathfrak{c}(j)) \cdot \exp\left( S_{|\mathfrak{c}(j)|}(\mathbf{t}) \cdot \pi \im \right) \right]
		= \big( \rho \circ \Phi_\mathfrak{c}(\mathbf{t}) \big)_j.
	\end{align*}
\end{proof}

\begin{corollary}
	\label{cor:vertices_c_transpose}
	Let $\mathfrak{c}$ be an $(m,n)$-cluster structure. For all $1 \leq k \leq m$, $\mathbf{v}^{(k)}(\rho \cdot \mathfrak{c}) = \mathbf{v}^{(m+1-k)}(\mathfrak{c})$ and, thus, $V(\rho \cdot \mathfrak{c}) = V(\mathfrak{c})$.
\end{corollary}
\begin{proof}
	Notice that $\mathbf{t}^{(k)} = \rho_{\Delta}(\mathbf{t}^{(m+1-k)})$, so
	\begin{equation*}
		\mathbf{v}^{(k)}(\rho \cdot \mathfrak{c}) = \Phi_{\rho \cdot \mathfrak{c}}(\mathbf{t}^{(k)}) = \rho \cdot \Phi_{\mathfrak{c}}(\mathbf{t}^{(m+1-k)}) = \rho(\mathbf{v}^{(m+1-k)}(\mathfrak{c})) = \mathbf{v}^{(m+1-k)}(\mathfrak{c}).
	\end{equation*}
\end{proof}

\begin{corollary}
	\label{cor:c_transpose}
	For any $\mathbf{x} \in \{1\} \times \T^{n-1}$, $\rho \cdot \mathfrak{c}_{\mathbf{x}} = \mathfrak{c}_{\rho(\mathbf{x})}$.
\end{corollary}
\begin{proof}
	Suppose that $\mathfrak{c}_{\mathbf{x}}$ is an $(m,n)$-cluster structure. By Proposition \ref{prop:cluster_structure_x}, there exists $\mathbf{t} \in \Delta_{m-1}$ such that $t_i > 0$ for all $1 \leq i \leq m$ and $\mathbf{x} = \Phi_{\mathfrak{c}}(\mathbf{t})$. Let $I := \{1, \dots, m\}$. By Proposition \ref{prop:c_transpose}, $\rho(\mathbf{x}) = \rho \circ \Phi_{\mathfrak{c}_\mathbf{x}}(\mathbf{t}) = \Phi_{\rho \cdot \mathfrak{c}_\mathbf{x}}(\rho_\Delta(\mathbf{t}))$. Since the $i$-th component of $\rho_\Delta(\mathbf{t})$ is positive if and only if $i \in I$, Lemma \ref{lemma:cluster_boundary} yields $\mathfrak{c}_{\rho(\mathbf{x})} = (\rho \cdot \mathfrak{c}_{\mathbf{x}})_I = \rho \cdot \mathfrak{c}_{\mathbf{x}}$.
\end{proof}

\subsection{The geometric realization of $\St_n(\Sp^1)$ is homeomorphic to $\Kn_n(\Sp^1)$.}
\label{subsec:geometric_realization}
Let's review the terminology from Section 1.3 of \cite{borsuk-ulam}. A finite set $\{\mathbf{v}_1, \dots, \mathbf{v}_m\} \subset \R^d$ is \emph{affinely independent} if the only $a_1, \dots, a_m \in \R$ such that $a_1 \mathbf{v}_1 + \cdots a_m \mathbf{v}_m = 0$ and $a_1 + \cdots + a_m = 0$ are $a_1 = \cdots = a_m = 0$. An \emph{$m$-simplex} $A \subset \R^d$ is the convex hull of a finite affinely independent set $V(A) \subset \R^d$ with $m+1$ points, and a face of $A$ is the convex hull of a subset of $V(A)$. A \emph{vertex} of $A$ is any point in $V(A)$. A nonempty family $\Delta$ of simplices is a \emph{geometric simplicial complex} if every face of $A \in \Delta$ is in $\Delta$ and for any $A, B \in \Delta$, $A \cap B$ is a face of both $A$ and $B$. The \emph{polyhedron} of $\Delta$ is $\|\Delta\| := \bigcup_{A \in \Delta} A \subset \R^d$. Every geometric simplicial complex has an associated abstract simplicial complex given by $K_\Delta := \{V(A) \mid A \in \Delta \}$, and any geometric realization $|K_\Delta|$ of $K_\Delta$ is homeomorphic to $\|\Delta\|$.\\
\indent Our objective in this section is to prove Theorem \ref{thm:geometric_realization_intro} by showing that
\begin{equation*}
	\Delta := \{ \conv(D_n \circ V(\mathfrak{c})) \mid \mathfrak{c} \text{ is an $(m,n)$-cluster structure} \}
\end{equation*}
is a geometric simplicial complex such that $\Kn_n(\Sp^1) = \|\Delta\| \cong |\St_n(\Sp^1)|$. The homeomorphism $\|\Delta\| \cong |\St_n(\Sp^1)|$ follows easily because $D_n$ induces a simplicial isomorphism from $\St_n(\Sp^1)$ to $K_\Delta = \{D_n \circ V(\mathfrak{c}) \mid \mathfrak{c} \text{ is an $(m,n)$-cluster structure}\}$. Since $\St_n(\Sp^1)$ is already an abstract simplicial complex by Corollary \ref{cor:cluster_boundary}, the only thing that we have to verify is that $D_n$ is injective on the vertex set of $\St_n(\Sp^1)$. Indeed, any two distinct $\mathbf{v}, \mathbf{w} \in V(\St_n(\Sp^1))$ have $v_1 = w_1 = 1$ and $v_i \neq w_i$ for some $2 \leq i \leq n$, so the $(1,i)$ entries of $D_n(\mathbf{v})$ and $D_n(\mathbf{w})$ are different. Thus, we have:

\begin{lemma}
	\label{lemma:St_n_Delta_isomorphic}
	$\St_n(\Sp^1)$ and $K_\Delta$ are isomorphic abstract simplicial complexes.
\end{lemma}

\indent Of course, the bulk of the work lies on the verification that $\Delta$ is a geometric simplicial complex. Let's start by proving that $\conv(D_n \circ V(\mathfrak{c}))$ is a simplex.

\begin{corollary}
	\label{cor:cluster_convex_hull}
	For any $(m,n)$-cluster structure $\mathfrak{c}$, $D_n \circ V(\mathfrak{c}) \subset \R^{n \times n}$ is affinely independent. Hence, $\conv(D_n \circ V(\mathfrak{c}))$ is an $(m-1)$-simplex.
\end{corollary}
\begin{proof}
	Notice that a set $V := \{\mathbf{v}_1, \dots, \mathbf{v}_m \} \subset \R^d$ is affinely independent if and only if the map $F:\Delta_{m-1} \to \conv(V)$ given by $F(\mathbf{t}) := t_1 \mathbf{v}_1 + \cdots t_m \mathbf{v}_m$ is injective. Proposition \ref{prop:cluster_convex_hull} shows that the composition $D_n \circ \Phi_{\mathfrak{c}}:\Delta_{m-1} \to \conv(D_n \circ V(\mathfrak{c}))$ has this form, and Corollary \ref{cor:Phi_c_injective} implies that it is injective. Hence, $D_n \circ V(\mathfrak{c})$ is affinely independent.
\end{proof}

The Proposition above verifies that $\Delta$ is a collection of simplices. Faces of simplices are simplices because we already know that $\St_n(\Sp^1)$ is an abstract simplicial complex. To verify that simplices intersect on their faces, we'll prove that there is a unique minimal simplex in $\Delta$ that contains any given matrix $M \in \Kn_n(\Sp^1)$.

\begin{prop}
	\label{prop:clusters_same_vertices}
	Let $\mathbf{x} \in \{1\} \times \T^{n-1}$. Let $\mathfrak{c}$ be an $(m,n)$-cluster structure such that $V(\mathfrak{c}) = V(\mathfrak{c}_{\mathbf{x}})$. Then $\mathfrak{c} = \mathfrak{c}_{\mathbf{x}}$ or $\mathfrak{c} = \mathfrak{c}_{\rho(\mathbf{x})}$.
\end{prop}
\begin{proof}
	Observe that a cluster structure has $m$ degrees of freedom if and only if it has $m$ vertices. Since $\mathfrak{c}$ has $m$ degrees of freedom, $V(\mathfrak{c}_{\mathbf{x}}) = V(\mathfrak{c})$ has $m$ points, so $\mathfrak{c}_{\mathbf{x}}$ has $m$ degrees of freedom. By Proposition \ref{prop:cluster_structure_x}, there exists an interior point $\mathbf{t}$ of $\Delta_{m-1}$ such that $\mathbf{x} = \Phi_{\mathfrak{c}_{\mathbf{x}}}(\mathbf{t})$. Then $D_n(\mathbf{x})$ is an interior point of $\conv(D_n \circ V(\mathfrak{c}_{\mathbf{x}})) = \conv(D_n \circ V(\mathfrak{c}))$, so by Proposition \ref{prop:cluster_convex_hull}, there exists an interior point $\mathbf{s} \in \Delta_{m-1}$ such that $D_n(\mathbf{x}) = D_n(\Phi_{\mathfrak{c}}(\mathbf{s}))$. Paraphrasing Proposition \ref{prop:quotient_of_torus} gives that $\Phi_{\mathfrak{c}}(\mathbf{s})$ is either $\mathbf{x}$ or $\rho(\mathbf{x})$. In the first case, Lemma \ref{lemma:cluster_boundary} applies with $I = \{1, \dots, m\}$ ($\mathbf{s}$ is an interior point of $\Delta_{m-1}$ means that $s_i > 0$ for all $i$), so $\mathfrak{c}_{\mathbf{x}} = \mathfrak{c}_I = \mathfrak{c}$. The second case yields $\mathfrak{c}_{\rho(\mathbf{x})} = \mathfrak{c}$ analogously.
\end{proof}

\begin{prop}
	\label{prop:minimal_simplex}
	Let $M \in \Kn_n(\Sp^1)$ and $\mathbf{x} \in \{1\} \times \T^{n-1}$ such that $M = D_n(\mathbf{x})$. If $M \in \conv(D_n \circ V(\mathfrak{c}))$ for some cluster structure $\mathfrak{c}$, then $V(\mathfrak{c}_{\mathbf{x}}) \subset V(\mathfrak{c})$. As a consequence, $\conv(D_n \circ V(\mathfrak{c}_{\mathbf{x}}))$ is the unique minimal simplex of $\Delta$ that contains $M$.
\end{prop}
\begin{proof}
	Paraphrasing Proposition \ref{prop:quotient_of_torus} shows that any $\mathbf{y} \in \{1\} \times \T^{n-1}$ such that $D_n(\mathbf{y}) = M$ has to be either $\mathbf{x}$ or $\rho(\mathbf{x})$. Corollaries \ref{cor:vertices_c_transpose} and \ref{cor:c_transpose} imply that $V(\mathfrak{c}_{\rho(\mathbf{x})}) = V(\rho \cdot \mathfrak{c}_{\mathbf{x}}) = V(\mathfrak{c}_\mathbf{x})$, so the choice between $\mathbf{x}$ and $\rho(\mathbf{x})$ is immaterial. Moreover, if $\mathfrak{c}$ is an arbitrary $(m,n)$-cluster structure such that $M \in \conv(D_n \circ V(\mathfrak{c}))$, Proposition \ref{prop:cluster_convex_hull} gives $\mathbf{t} \in \Delta_{m-1}$ such that $M = D_n \circ \Phi_{\mathfrak{c}}(\mathbf{t})$. We assume $\mathbf{x} = \Phi_{\mathfrak{c}}(\mathbf{t})$ without loss of generality because, if $\Phi_{\mathfrak{c}}(\mathbf{t}) = \rho(\mathbf{x})$, Proposition \ref{prop:c_transpose} and Corollary \ref{cor:c_transpose} allow us to replace $\mathfrak{c}$ with $\rho \cdot \mathfrak{c}$. Then Lemma \ref{lemma:cluster_boundary} implies that $\mathfrak{c}_{\mathbf{x}}$ has the form $\mathfrak{c}_I$ for some $I \subset \{1, \dots, m\}$, so Corollary \ref{cor:cluster_boundary} yields $V(\mathfrak{c}_{\mathbf{x}}) = V(\mathfrak{c}_I) \subset V(\mathfrak{c}) = V(\rho \cdot \mathfrak{c})$. Hence, $\conv(D_n \circ V(\mathfrak{c}_{\mathbf{x}}))$ is the unique minimal simplex that contains $M$.
\end{proof}

\begin{prop}
	\label{prop:intersection_of_simplices}
	Let $\mathfrak{c}$ be an $(m,n)$-cluster structure and $\mathfrak{c}'$, an $(m',n)$-cluster structure. Then $\conv(D_n \circ V(\mathfrak{c})) \cap \conv(D_n \circ V(\mathfrak{c}')) = \conv(D_n \circ V(\mathfrak{c}) \cap D_n \circ V(\mathfrak{c}'))$.
\end{prop}
\begin{proof}
	Since $V(\mathfrak{c}) \cap V(\mathfrak{c}')$ is contained in both $V(\mathfrak{c})$ and $V(\mathfrak{c}')$, the inclusion
	\begin{equation*}
		\conv(D_n \circ V(\mathfrak{c}) \cap D_n \circ V(\mathfrak{c}')) \subset \conv(D_n \circ V(\mathfrak{c})) \cap \conv(D_n \circ V(\mathfrak{c}'))
	\end{equation*}
	is immediate. Conversely, let $M \in \conv(D_n \circ V(\mathfrak{c})) \cap \conv(D_n \circ V(\mathfrak{c}'))$, and let $\mathbf{x} \in \{1\} \times \T^{n-1}$ such that $M = D_n(\mathbf{x})$. By Proposition \ref{prop:minimal_simplex}, $V(\mathfrak{c}_{\mathbf{x}}) \subset V(\mathfrak{c}) \cap V(\mathfrak{c}')$, so $M \in \conv(D_n \circ V(\mathfrak{c}_{\mathbf{x}})) \subset \conv(D_n \circ V(\mathfrak{c}) \cap D_n \circ V(\mathfrak{c}'))$.
\end{proof}

\begin{theorem}
	\label{thm:geometric_realization}
        The collection
        \begin{equation*}
    	\Delta = \{ \conv(D_n \circ V(\mathfrak{c})) \mid \mathfrak{c} \text{ is an $(m,n)$-cluster structure} \}
        \end{equation*}
	is a geometric simplicial complex such that $\| \Delta \| = \Kn_n(\Sp^1)$.
\end{theorem}
\begin{proof}
	By Corollary \ref{cor:cluster_convex_hull}, $\Delta$ is a collection of simplices. Given an $(m,n)$-cluster structure $\mathfrak{c}$, any subset of $V(\mathfrak{c})$ is the vertex set of some cluster structure by Corollary \ref{cor:cluster_boundary}, so every face of the simplex $\conv(D_n \circ V(\mathfrak{c}))$ is also a simplex of $\Delta$. Finally, Proposition \ref{prop:intersection_of_simplices} implies that the intersection of any pair of simplices of $\Delta$ is again a simplex. These three conditions verify that $\Delta$ is a geometric simplicial complex. To show that $\| \Delta \| = \Kn_n(\Sp^1)$, notice that the simplices of $\Delta$ are subsets of $\Kn_n(\Sp^1)$ by definition, so $\| \Delta \| \subset \Kn_n(\Sp^1)$. Conversely, for any $M \in \Kn_n(\Sp^1)$, there exists a unique minimal simplex of $\Delta$ that contains $M$ by Proposition \ref{prop:minimal_simplex}, so $\Kn_n(\Sp^1) \subset \| \Delta \|$.
\end{proof}

Thus, we obtain the main result of this section.
\GeometricRealizationStn*
\begin{proof}
	By Lemma \ref{lemma:St_n_Delta_isomorphic} and Theorem \ref{thm:geometric_realization}, $|\St_n(\Sp^1)| \cong \|\Delta\| = \Kn_n(\Sp^1)$.
\end{proof}

\subsection{Face numbers of $\St_n(\Sp^1)$ and Euler characteristic.}
\label{sec:combinatorics_St_n}

\indent Corollary \ref{cor:vertices_c_transpose} shows that for every $(m,n)$-cluster structure $\mathfrak{c}$, $\rho \cdot \mathfrak{c}$ has the same vertex set as $\mathfrak{c}$, while Proposition \ref{prop:clusters_same_vertices} shows that no other cluster structure $\mathfrak{c}'$ has the same vertices as $\mathfrak{c}$. We use this information to count the number $f(n,k)$ of $k$-simplices of $\St_n(\Sp^1)$. Let $S(n,k)$ be the Stirling numbers of the second kind. These are the number of ways that $\{1, \dots, n\}$ can be partitioned into $k$ unordered sets. The $S(n,k)$ are defined for $0 \leq k \leq n$ and satisfy the recurrence relation
\begin{equation*}
	S(n+1,k) = kS(n,k) + S(n,k-1) \text{ for } 0 < k < n,
\end{equation*}
with initial conditions $S(n,n)=1$ for $n \geq 0$ and $S(n,0) = S(0,n) = 0$ for $n>0$.

\begin{prop}
	\label{prop:Euler_char}
	$f(1,0) = 1$ and if $n \geq 2$, $f(n,0) = 2^{n-1}$ and $f(n,m) = 2^{n-2} \cdot m! \cdot S(n,m+1)$ for $1 \leq m \leq n-1$.
\end{prop}
\begin{proof}
	Recall that the $m$-simplices of $\St_n(\Sp^1)$ are determined by $(m+1,n)$-cluster structures (see Definitions \ref{def:cluster_structure_vertices} and \ref{def:state_complex}). If $n=1$, there exists a unique $(1,1)$-cluster structure $\mathfrak{c}:\{1\} \to \{\pm 1\}$ because $\mathfrak{c}(1) = +1$ by Definition \ref{def:cluster_structure}. Hence, $f(1,0)=1$. If $n \geq 2$, Remarks \ref{rmk:vertices_of_c} and \ref{rmk:vertices_of_St_n} imply that the vertex set of $\St_n(\Sp^1)$ is the set of points $\mathbf{v} \in \T^n$ such that $v_1 = +1$ and $v_i = \pm 1$ for $2 \leq i \leq n$. There are $2^{n-1}$ such points, so $f(n,0) = 2^{n-1}$.\\
	\indent Suppose $n \geq 2$ and $1 \leq m \leq n-1$. Let's count the number $C(n,m+1)$ of $(m+1,n)$-cluster structures. Observe that a surjective function $|\mathfrak{c}|:\{1, \dots, n\} \to \{1, \dots, m+1\}$ assigns each $1 \leq i \leq n$ into one of $(m+1)$ ordered clusters. Normally, there would be $(m+1)! \cdot S(n,m+1)$ possibilities for $|\mathfrak{c}|$, but Definition \ref{def:cluster_structure} requires $\mathfrak{c}(1)=1$. In other words, the first cluster is fixed, so we have $m! \cdot S(n,m+1)$ possible choices for $|\mathfrak{c}|$. To obtain $\mathfrak{c}$ from $|\mathfrak{c}|$, we have to choose signs for $\mathfrak{c}(2), \dots, \mathfrak{c}(n)$, so $C(n,m+1) = 2^{n-1} \cdot m! \cdot S(n,m+1)$. Since the simplices $V(\mathfrak{c}), V(\mathfrak{c}') \in \St_n(\Sp^1)$ are equal if and only if $\mathfrak{c}'=\mathfrak{c}$ or $\mathfrak{c}'=\rho \cdot \mathfrak{c}$, we get $f(n,m) = C(n,m+1)/2 = 2^{n-2} \cdot m! \cdot S(n,m+1)$.
\end{proof}

\begin{remark}
	Since $S(n,2)$ equals $2^{n-1}-1$ for $n \geq 2$, the number of edges in $\St_n(\Sp^1)$ is $f(n,1) = 2^{n-2} S(n,2) = \frac{1}{2} 2^{n-1} (2^{n-1}-1) = \binom{2^{n-1}}{2}$. As a consequence, $\St_n(\Sp^1)$ contains an edge between every pair of vertices when $n \geq 2$.
\end{remark}

As an application of Proposition \ref{prop:Euler_char}, we calculate the Euler characteristic of $\Kn_n(\Sp^1)$ for any $n$.
\begin{prop}
	$\chi(\St_1(\Sp^1)) = 1$ and, for $n \geq 2$, $\chi(\St_n(\Sp^1)) = 2^{n-2}$.
\end{prop}
\begin{proof}
	
	Observe that
	\begin{align*}
		\chi(\St_n(\Sp^1))
		&= \sum_{m=0}^{n-1} (-1)^m f(n,m)
		= 2^{n-1} + \sum_{m=1}^{n-1} (-1)^m 2^{n-2} m! \cdot S(n,m+1) \\
		&= 2^{n-2} + 2^{n-2} \sum_{m=0}^{n-1} (-1)^m m! \cdot S(n,m+1).
	\end{align*}
	
	Using the recurrence relation of $S(n,m)$ gives
	\begin{align*}
		\sum_{m=0}^{n-1} & (-1)^m m! \cdot S(n,m+1)
		= \sum_{m=0}^{n-2} (-1)^m m! \big[ (m+1) \cdot S(n-1,m+1) + S(n-1,m) \big] \\
		&\phantom{(-1)^m m! \cdot S(n,m+1) = \sum_{m=0}^{n-2}} + (-1)^{n-1} (n-1)! \cdot S(n,n)\\
		&=  \sum_{m=0}^{n-2} (-1)^m (m+1)! \cdot S(n-1,m+1) + \sum_{m=0}^{n-1} (-1)^m m! \cdot S(n-1,m) \\
		&=  \sum_{m=0}^{n-2} (-1)^m (m+1)! \cdot S(n-1,m+1) + \sum_{m=1}^{n-1} (-1)^m m! \cdot S(n-1,m) + 0! \cdot S(n-1,0) \\
		&=  \sum_{m=0}^{n-2} (-1)^m (m+1)! \cdot S(n-1,m+1) + \sum_{m=0}^{n-2} (-1)^{m+1} (m+1)! \cdot S(n-1,m+1) = 0.
	\end{align*}
	Thus, $\chi (\St_n(\Sp^1)) = 2^{n-2}$.
\end{proof}

\begin{remark}
	The Euler characteristic of $\Kn_n(\Sp^1)$ can be computed via Theorem \ref{thm:main} and coincides with that of $\St_n(\Sp^1)$. This is to be expected in light of the homeomorphism $|\St_n(\Sp^1)| \cong \Kn_n(\Sp^1)$ given in Theorem \ref{thm:geometric_realization}.
\end{remark} 
\section{Connection with elliptopes}
\label{sec:elliptopes}
Let $\mathcal{S}^n$ be the set of $n$-by-$n$ symmetric matrices. Define the \emph{$n$-th elliptope} as:
\begin{equation*}
	\mathcal{E}_n := \{M \in \mathcal{S}^n : \ M \text{ is positive semidefinite and } M_{ii}=1 \text{ for } i=1,\dots,n \}.
\end{equation*}
Many problems in convex optimization can be seen as the optimization of an objective function over the set $\mathcal{E}_n$. One case where these problems are interesting is as a semidefinite relaxation of possibly NP-hard problems such as binary quadratic optimizations. In other words, it is possible to obtain approximate solutions to a binary quadratic optimization using a semidefinite program in a much reasonable time. See \cite{semidefinite-optimization-ag} for a more complete treatment on optimization, including elliptopes. In particular, Section 2.2.2 contains more thorough examples of semidefinite relaxations.\\
\indent There is a connection between $\mathcal{E}_n$ and the curvature sets $\Kn_n(\Sp^m)$. Given two points $x, y \in \Sp^{m}$ thought of as column vectors in $\R^{m+1}$, the geodesic distance in $\Sp^m$ is given by $d_m(x, y) := \arccos(x^\intercal \cdot y)$. In fact, for $\mathbf{x} \in (\Sp^m)^n$ (thought of as an $n$-by-$(m+1)$ matrix with columns $x_i \in \Sp^{m} \subset \R^{m+1}$), then $\left( D_n(\mathbf{x}) \right)_{ij} = \arccos(x_i^\intercal \cdot x_j)$. Then the matrix obtained by applying $\cos(\cdot)$ entry-wisely to $D_n(\mathbf{x})$ is the Gram matrix $\mathbf{x}^\intercal \cdot \mathbf{x}$, which is positive semi-definite. Hence, $\cos(D_n(\mathbf{x})) \in \mathcal{E}_n$. Conversely, every $M \in \mathcal{E}_n$ is a Gram matrix of a set of points in $\R^{m+1}$ where $m = \operatorname{rank}(M)$ and $0 \leq m \leq n-1$. Since $M_{ii}=1$, these points actually belong to $\Sphere{m}$. We then have the following result.
\begin{prop}
	\label{prop:elliptope_filtration}
	The chain of inclusions $\Sphere{1} \subset \Sphere{2} \subset \cdots \subset  \Sphere{n-1}$ induces a filtration
	\begin{equation*}
		\cos(\Kn_n(\Sphere{1})) \subset \cdots \subset \cos(\Kn_n(\Sphere{n-1})) = \mathcal{E}_n
	\end{equation*}
\end{prop}

This proposition becomes suggestive in light of the computation of the homology groups of $\Kn_n(\Sp^1)$ (Theorem \ref{thm:main}): it implies that the convex set $\mathcal{E}_n$ has a subset $\cos(\Kn_n(\Sp^1)) \cong \Kn_n(\Sp^1)$ which not only has non-trivial homology but also exhibits torsion. Additionally, M\'{e}moli showed that $\Kn_3(\Sp^2)$ is the convex hull of $\Kn_3(\Sp^1)$ \cite{mem12}. In the case $n=3$, Theorem \ref{prop:elliptope_filtration} means that $\mathcal{E}_3$ is the convex hull of $\cos(\Kn_3(\Sp^1))$. In fact, we can show a similar statement. Let $\Sp^{n}_E$ be the $n$-sphere in $\R^{n+1}$ equipped with the Euclidean metric, and denote with $D_n^{m}: (\Sp^m_E)^n \to \Kn_n(\Sp^m_E)$ the distance matrix map (see Section \ref{sec:intro}).
\begin{prop}
	\label{prop:Kn_convex_hull}
	$\Kn_{n+1}(\Sp_E^{n})$ is the cone joining $\Kn_{n+1}(\Sp_E^{n-1})$ to the origin in $\R^{(n+1) \times (n+1)}$.
\end{prop}
\begin{proof}
	Let $\mathbf{x} \in (\Sp^{n}_E)^{n+1}$. Any $n+1$ points in $\R^{n+1}$ generate an $m$-hyperplane, with $m < n+1$, that intersects $\Sp^{n}_E$ in a rescaled sphere $\lambda \cdot \Sp^{m-1}_E$ with $0 < \lambda \leq 1$. Then, $D_{n+1}^{n}(\mathbf{x}) \in \lambda \cdot \Kn_{n+1}(\Sp^{m-1}_E) \subset \lambda \cdot \Kn_{n+1}(\Sp^{n-1}_E)$. Thus, $\Kn_{n+1}(\Sp^{n}_E) = \bigcup_{0 < \lambda \leq 1} \lambda \cdot \Kn_{n+1}(\Sp^{n-1}_E)$.
\end{proof}

Let $d_{m,E}$ be the Euclidean metric on $\Sp_E^m$, and observe that the function $f_E(d) := 2\arcsin(d/2)$ satisfies $d_{m,E} = f_E \circ d_m$ and $\Kn_n(\Sp_E^m) = f_E\big( \Kn_n(\Sp^m) \big)$. Propositions \ref{prop:elliptope_filtration} and \ref{prop:Kn_convex_hull} imply that $\mathcal{E}_n$ is the convex hull of $\cos\big(f_E^{-1}(\Kn_{n}(\Sp^{n-2})) \big)$. It seems  interesting to explore these and other consequences of the connection between elliptopes and curvature sets.

\bibliographystyle{alpha}

\begin{thebibliography}{99}
	\bibitem[AG02]{confspaces_factory}
	Aaron Abrams and Robert Ghrist.
	\newblock Finding topology in a factory: Configuration spaces.
	\newblock {\em The American Mathematical Monthly}, 109(2):140--150, 2002.
	
	\bibitem[AKM21]{confspaces_disks_infinite_strip}
	Hannah Alpert, Matthew Kahle, and Robert MacPherson.
	\newblock Configuration spaces of disks in an infinite strip.
	\newblock {\em Journal of Applied and Computational Topology}, 5(3):357--390,
	Sep 2021.
	
	\bibitem[BBK13]{confspaces_hard_spheres_morse}
	Yuliy Baryshnikov, Peter Bubenik, and Matthew Kahle.
	\newblock {Min-Type Morse Theory for Configuration Spaces of Hard Spheres}.
	\newblock {\em International Mathematics Research Notices}, 2014(9):2577--2592,
	02 2013.
	
	\bibitem[BCT89]{confspaces_homology}
	C.-F. Bödigheimer, F.~Cohen, and L.~Taylor.
	\newblock On the homology of configuration spaces.
	\newblock {\em Topology}, 28(1):111--123, 1989.
	
	\bibitem[Bot52]{ran-space-S1-old}
	Raoul Bott.
	\newblock On the third symmetric potency of ${S}^1$.
	\newblock {\em Fund. Math}, 39:264--268, 1952.
	
	\bibitem[BPT12]{semidefinite-optimization-ag}
	Grigoriy Blekherman, Pablo~A. Parrilo, and Rekha~R. Thomas.
	\newblock {\em Semidefinite Optimization and Convex Algebraic Geometry}.
	\newblock Society for Industrial and Applied Mathematics, Philadelphia, PA,
	2012.
	
	\bibitem[CGKM12]{confspaces_hard_disks_computation}
	Gunnar Carlsson, Jackson Gorham, Matthew Kahle, and Jeremy Mason.
	\newblock Computational topology for configuration spaces of hard disks.
	\newblock {\em Phys. Rev. E}, 85:011303, Jan 2012.
	
	\bibitem[Coh95]{confspaces_lie_algebras}
	F.R. Cohen.
	\newblock On configuration spaces, their homology, and lie algebras.
	\newblock {\em Journal of Pure and Applied Algebra}, 100(1):19--42, 1995.
	
	\bibitem[Coh09]{confspaces_introduction}
	Frederick~R. Cohen.
	\newblock {\em Introduction to configuration spaces and their applications},
	pages 183--261.
	\newblock 2009.
	
	\bibitem[FT05]{confspaces_unordered_cohomology}
	Yves Félix and Daniel Tanré.
	\newblock The cohomology algebra of unordered configuration spaces.
	\newblock {\em Journal of the London Mathematical Society}, 72(2):525--544,
	2005.
	
	\bibitem[Ghr]{confspaces_braids_robotics}
	Robert Ghrist.
	\newblock {\em Configuration spaces, braids, and robotics}, pages 263--304.
	
	\bibitem[Gro07]{gro99}
	Misha Gromov.
	\newblock {\em Metric Structures for Riemannian and non-Riemannian Spaces}.
	\newblock Modern Birkh{\"a}user Classics. Birkh{\"a}user Boston Inc, Boston,
	MA, 2007.
	
	\bibitem[Kri94]{confspaces_rational_homotopy}
	Igor Kriz.
	\newblock On the rational homotopy type of configuration spaces.
	\newblock {\em Annals of Mathematics}, 139(2):227--237, 1994.
	
	\bibitem[Mat08]{borsuk-ulam}
	Ji\v{r}\'{i} Matou\v{s}ek.
	\newblock {\em Using the Borsuk-Ulam Theorem}.
	\newblock Universitext. Springer Berlin, Heidelberg, 2008.
	
	\bibitem[M{\'e}m12]{mem12}
	Facundo M{\'e}moli.
	\newblock Some properties of {{G}romov-{H}ausdorff} distances.
	\newblock {\em Discrete {\&} Computational Geometry}, 48(2):416--440, September
	2012.
	
	\bibitem[Nap03]{confspaces_surfaces_cohomology}
	Fabien Napolitano.
	\newblock On the cohomology of configuration spaces on surfaces.
	\newblock {\em Journal of the London Mathematical Society}, 68(2):477--492,
	2003.
	
	\bibitem[Sal04]{confspaces_sphere_loop_spaces}
	Paolo Salvatore.
	\newblock Configuration spaces on the sphere and higher loop spaces.
	\newblock {\em Mathematische Zeitschrift}, 248(3):527--540, Nov 2004.
	
	\bibitem[Sch18]{confspaces_sphere_integral_cohomology}
	Christoph Schiessl.
	\newblock Integral cohomology of configuration spaces of the sphere.
	\newblock {\em Homology, Homotopy and Applications}, 21(1):283--302, 2018.
	
	\bibitem[Tot96]{confspaces_varieties}
	Burt Totaro.
	\newblock Configuration spaces of algebraic varieties.
	\newblock {\em Topology}, 35(4):1057--1067, 1996.
	
	\bibitem[Tuf02]{ran-space-S1}
	Christopher Tuffley.
	\newblock Finite subset spaces of ${S}^1$.
	\newblock {\em Algebr. Geom. Topol}, 2:1119--1145, 2002.
	
	\bibitem[Tuf03]{ran-space-surfaces}
	Christopher Tuffley.
	\newblock Finite subset spaces of closed surfaces.
	\newblock {\em arXiv preprint arxiv:1909.12313}, 2003.
\end{thebibliography}

\newpage
\appendix
\section{Explicit computation of $\St_3(\Sp^1)$.}
\begin{example}
	\label{ex:state_complex}
	We now show an explicit computation of $\St_3(\Sp^1)$. To do that, we show all $(m,3)$-cluster structures with $1 \leq m \leq 3$ and show a typical element of $\conv(D_n \circ V(\mathfrak{c})) \subset \Kn_{3}(\Sp^1)$. We also include the Hasse diagram of $\St_3(\Sp^1)$, from where it can be seen that $|\St_3(\Sp^1)| \simeq \Sp^2$. For convenience, we denote the points $+1,-1 \in \Sp^1$ with their signs only, and a cluster structure $\mathfrak{c}$ as a vector $(\mathfrak{c}(1), \mathfrak{c}(2), \mathfrak{c}(3))$.
	
	\begin{figure}[h]
		\centering
		\begin{equation*}
			\begin{tikzcd}
				& f_1 = f_2		\arrow[dash]{ddl}	\arrow[dash]{dd}	\arrow[dash]{ddr}
				& f_3 = f_4		\arrow[dash]{ddll}	\arrow[dash]{ddr}	\arrow[dash]{ddrr}
				& f_5 = f_6		\arrow[dash]{ddll}	\arrow[dash]{dd}	\arrow[dash]{ddrr}
				& f_7 = f_8		\arrow[dash]{ddll}	\arrow[dash]{dd}	\arrow[dash]{ddr}
				\\ \\
				e_1 = e_2			\arrow[dash]{ddr}	\arrow[dash]{ddrr}
				& e_3 = e_4			\arrow[dash]{dd}	\arrow[dash]{ddrr}
				& e_5 = e_6			\arrow[dash]{dd}	\arrow[dash]{ddr}
				& e_7 = e_8			\arrow[dash]{ddll}	\arrow[dash]{ddr}
				& e_9 = e_{10}		\arrow[dash]{ddll}	\arrow[dash]{dd}
				& e_{11} = e_{12}	\arrow[dash]{ddll}	\arrow[dash]{ddl}
				\\ \\
				& v_1 \arrow[d, phantom, sloped, "="]
				& v_2 \arrow[d, phantom, sloped, "="]
				& v_3 \arrow[d, phantom, sloped, "="]
				& v_4 \arrow[d, phantom, sloped, "="]
				\\
				& (+--)
				& (++-)
				& (+-+)
				& (+++)
			\end{tikzcd}
		\end{equation*}
		\caption{Hasse diagram of $\St_3(\Sp^1)$.}
	\end{figure}

	\begin{longtable}{|c|c|c|c|c|}
		\caption{The 2-simplices of $\St_3(\Sp^1)$. No pair of points is equal or antipodal. Notice that, for all $k$, $\alpha \circ \pi(\mathbf{x})$ is equal in rows $f_{2k-1}$ and $f_{2k}$, while $\mathfrak{c}$ in row $2k$ equals $\mathfrak{c}^\intercal$ in row $2k-1$.}\\
		
		\hline
		Name & $\mathfrak{c}$ & $V(\mathfrak{c})$ & $\mathbf{x} \in \Im(\Phi_{\mathfrak{c}})$ & $D_3(\mathbf{x}) = \sum_{i=1}^{3} t_i D_3(\mathbf{v}^{(i)}(\mathfrak{c}))$\\
		\hline
		\endhead

		$f_1$
		&
		\begin{minipage}{0.15\linewidth}
			\centering
			$(+1,-2,+3)$
		\end{minipage}
		&
		\begin{minipage}{0.15\linewidth}
			\centering
			$(++-)$ \\
			$(+--)$ \\
			$(+-+)$
		\end{minipage}
		&
		\begin{minipage}{0.26\linewidth}
			\centering
			\drawCircle{0}{240}{105}{$x_1$}{$x_2$}{$x_3$}
		\end{minipage}
		&
		\begin{minipage}{0.30\linewidth}
			\centering
			$
			\begin{pmatrix}
				0		& t_2+t_3	& t_1+t_2 \\
				t_2+t_3	& 0 		& t_1+t_3 \\
				t_1+t_2	& t_1+t_3 	& 0
			\end{pmatrix}
			$
		\end{minipage}
		\\ \hline

		$f_2$
		&
		\begin{minipage}{0.15\linewidth}
			\centering
			$(+1,+3,-2)$
		\end{minipage}
		&
		\begin{minipage}{0.15\linewidth}
			\centering
			$(+-+)$ \\
			$(+--)$ \\
			$(++-)$
		\end{minipage}
		&
		\begin{minipage}{0.26\linewidth}
			\centering
			\drawCircle{0}{120}{255}{$x_1$}{$x_2$}{$x_3$}
		\end{minipage}
		&
		\begin{minipage}{0.30\linewidth}
			\centering
			$
			\begin{pmatrix}
				0		& t_1+t_2	& t_2+t_3 \\
				t_1+t_2	& 0 		& t_1+t_3 \\
				t_2+t_3	& t_1+t_3 	& 0
			\end{pmatrix}
			$
		\end{minipage}
		\\ \hline

		$f_3$
		&
		\begin{minipage}{0.15\linewidth}
			\centering
			$(+1,+2,+3)$
		\end{minipage}
		&
		\begin{minipage}{0.15\linewidth}
			\centering
			$(+--)$ \\
			$(++-)$ \\
			$(+++)$
		\end{minipage}
		&
		\begin{minipage}{0.26\linewidth}
			\centering
			\drawCircle{0}{65}{115}{$x_1$}{$x_2$}{$x_3$}
		\end{minipage}
		&
		\begin{minipage}{0.30\linewidth}
			\centering
			$
			\begin{pmatrix}
				0		& t_1	& t_1+t_2 \\
				t_1		& 0 	& t_2 \\
				t_1+t_2	& t_2 	& 0
			\end{pmatrix}
			$
		\end{minipage}
		\\ \hline

		$f_4$
		&
		\begin{minipage}{0.15\linewidth}
			\centering
			$(+1,-3,-2)$
		\end{minipage}
		&
		\begin{minipage}{0.15\linewidth}
			\centering
			$(+++)$ \\
			$(++-)$ \\
			$(+--)$
		\end{minipage}
		&
		\begin{minipage}{0.26\linewidth}
			\centering
			\drawCircle{0}{295}{245}{$x_1$}{$x_2$}{$x_3$}
		\end{minipage}
		& 
		\begin{minipage}{0.30\linewidth}
			\centering
			$
			\begin{pmatrix}
				0		& t_3	& t_2+t_3 \\
				t_3		& 0 	& t_2 \\
				t_2+t_3	& t_2	& 0
			\end{pmatrix}
			$
		\end{minipage}
		\\ \hline

		$f_5$
		&
		\begin{minipage}{0.15\linewidth}
			\centering
			$(+1,+3,+2)$
		\end{minipage}
		&
		\begin{minipage}{0.15\linewidth}
			\centering
			$(+--)$ \\
			$(+-+)$ \\
			$(+++)$
		\end{minipage}
		&
		\begin{minipage}{0.26\linewidth}
			\centering
			\drawCircle{0}{115}{65}{$x_1$}{$x_2$}{$x_3$}
		\end{minipage}
		&
		\begin{minipage}{0.30\linewidth}
		\centering
		$
		\begin{pmatrix}
			0		& t_1+t_2	& t_1 \\
			t_1+t_2	& 0 		& t_2 \\
			t_1		& t_2		& 0
		\end{pmatrix}
		$
		\end{minipage}
		\\ \hline

		$f_6$
		&
		\begin{minipage}{0.15\linewidth}
			\centering
			$(+1,-2,-3)$
		\end{minipage}
		&
		\begin{minipage}{0.15\linewidth}
			\centering
			$(+++)$ \\
			$(+-+)$ \\
			$(+--)$
		\end{minipage}
		&
		\begin{minipage}{0.26\linewidth}
			\centering
			\drawCircle{0}{245}{295}{$x_1$}{$x_2$}{$x_3$}
		\end{minipage}
		&
		\begin{minipage}{0.30\linewidth}
			\centering
			$
			\begin{pmatrix}
				0		& t_2+t_3	& t_3 \\
				t_2+t_3	& 0 		& t_2 \\
				t_3		& t_2		& 0
			\end{pmatrix}
			$
		\end{minipage}
		\\ \hline

		$f_7$
		&
		\begin{minipage}{0.15\linewidth}
			\centering
			$(+1,-3,+2)$
		\end{minipage}
		&
		\begin{minipage}{0.15\linewidth}
			\centering
			$(++-)$ \\
			$(+++)$ \\
			$(+-+)$
		\end{minipage}
		&
		\begin{minipage}{0.26\linewidth}
			\centering
			\drawCircle{0}{295}{90}{$x_1$}{$x_2$}{$x_3$}
		\end{minipage}
		&
		\begin{minipage}{0.30\linewidth}
		\centering
		$
		\begin{pmatrix}
			0	& t_3		& t_1 \\
			t_3	& 0 		& t_1+t_3 \\
			t_1	& t_1+t_3	& 0
		\end{pmatrix}
		$
		\end{minipage}
		\\ \hline

		$f_8$
		&
		\begin{minipage}{0.15\linewidth}
			\centering
			$(+1,+2,-3)$
		\end{minipage}
		&
		\begin{minipage}{0.15\linewidth}
			\centering
			$(+-+)$ \\
			$(+++)$ \\
			$(++-)$
		\end{minipage}
		&
		\begin{minipage}{0.26\linewidth}
			\centering
			\drawCircle{0}{65}{270}{$x_1$}{$x_2$}{$x_3$}
		\end{minipage}
		&
		\begin{minipage}{0.30\linewidth}
		\centering
		$
		\begin{pmatrix}
			0	& t_1		& t_3 \\
			t_1	& 0 		& t_1+t_3 \\
			t_3	& t_1+t_3	& 0
		\end{pmatrix}
		$
		\end{minipage}
		\\ \hline
	\end{longtable}

	\begin{longtable}{|c|c|c|c|c|}
		\caption{The 1-simplices of $\St_3(\Sp^1)$. Two points are equal or antipodal to each other, while the third is not.}\\
		
		\hline
		Name & $\mathfrak{c}$ & $V(\mathfrak{c})$ & $\mathbf{x} \in \Im(\Phi_{\mathfrak{c}})$ & $D_3(\mathbf{x}) = \sum_{i=1}^{2} t_i D_3(\mathbf{v}^{(i)}(\mathfrak{c}))$\\
		\hline
		\endhead

		$e_1$
		&
		\begin{minipage}{0.15\linewidth}
			\centering
			$(+1,+2,-1)$
		\end{minipage}
		&
		\begin{minipage}{0.15\linewidth}
			\centering
			$(+--)$ \\
			$(++-)$
		\end{minipage}
		&
		\begin{minipage}{0.26\linewidth}
			\centering
			\drawCircle{0}{65}{180}{$x_1$}{$x_2$}{$x_3=-x_1$}
		\end{minipage}
		&
		\begin{minipage}{0.20\linewidth}
			\centering
			$
			\begin{pmatrix}
				0	& t_1	& \pi \\
				t_1	& 0 	& t_2 \\
				\pi	& t_2	& 0
			\end{pmatrix}
			$
		\end{minipage}
		\\ \hline

		$e_2$
		&
		\begin{minipage}{0.15\linewidth}
			\centering
			$(+1,-2,-1)$
		\end{minipage}
		&
		\begin{minipage}{0.15\linewidth}
			\centering
			$(++-)$ \\
			$(+--)$
		\end{minipage}
		&
		\begin{minipage}{0.26\linewidth}
			\centering
			\drawCircle{0}{295}{180}{$x_1$}{$x_2$}{$x_3=-x_1$}
		\end{minipage}
		&
		\begin{minipage}{0.20\linewidth}
		\centering
		$
		\begin{pmatrix}
			0	& t_2	& \pi \\
			t_2	& 0 	& t_1 \\
			\pi	& t_1	& 0
		\end{pmatrix}
		$
		\end{minipage}
		\\ \hline

		$e_3$
		&
		\begin{minipage}{0.15\linewidth}
			\centering
			$(+1,-1,+2)$
		\end{minipage}
		&
		\begin{minipage}{0.15\linewidth}
			\centering
			$(+--)$ \\
			$(+-+)$
		\end{minipage}
		&
		\begin{minipage}{0.26\linewidth}
			\centering
			\drawCircle{0}{180}{115}{$x_1$}{$x_2=-x_1$}{$x_3$}
		\end{minipage}
		&
		\begin{minipage}{0.20\linewidth}
		\centering
		$
		\begin{pmatrix}
			0	& \pi	& t_1 \\
			\pi	& 0 	& t_2 \\
			t_1	& t_2	& 0
		\end{pmatrix}
		$
		\end{minipage}
		\\ \hline

		$e_4$
		&
		\begin{minipage}{0.15\linewidth}
			\centering
			$(+1,-1,-2)$
		\end{minipage}
		&
		\begin{minipage}{0.15\linewidth}
			\centering
			$(+-+)$ \\
			$(+--)$
		\end{minipage}
		&
		\begin{minipage}{0.26\linewidth}
			\centering
			\drawCircle{0}{180}{245}{$x_1$}{$x_2=-x_1$}{$x_3$}
		\end{minipage}
		&
		\begin{minipage}{0.20\linewidth}
			\centering
			$
			\begin{pmatrix}
				0	& \pi	& t_2 \\
				\pi	& 0 	& t_1 \\
				t_2	& t_1	& 0
			\end{pmatrix}
			$
		\end{minipage}
		\\ \hline

		$e_5$
		&
		\begin{minipage}{0.15\linewidth}
			\centering
			$(+1,-2,+2)$
		\end{minipage}
		&
		\begin{minipage}{0.15\linewidth}
			\centering
			$(++-)$ \\
			$(+-+)$
		\end{minipage}
		&
		\begin{minipage}{0.26\linewidth}
			\centering
			\drawCircle{0}{295}{115}{$x_1$}{$x_2$}{$x_3=-x_2$}
		\end{minipage}
		&
		\begin{minipage}{0.20\linewidth}
			\centering
			$
			\begin{pmatrix}
				0	& t_2	& t_1 \\
				t_2	& 0 	& \pi \\
				t_1	& \pi	& 0
			\end{pmatrix}
			$
		\end{minipage}
		\\ \hline

		$e_6$
		&
		\begin{minipage}{0.15\linewidth}
			\centering
			$(+1,+2,-2)$
		\end{minipage}
		&
		\begin{minipage}{0.15\linewidth}
			\centering
			$(+-+)$ \\
			$(++-)$
		\end{minipage}
		&
		\begin{minipage}{0.26\linewidth}
			\centering
			\drawCircle{0}{65}{245}{$x_1$}{$x_2$}{$x_3=-x_2$}
		\end{minipage}
		&
		\begin{minipage}{0.20\linewidth}
			\centering
			$
			\begin{pmatrix}
				0	& t_1	& t_2 \\
				t_1	& 0 	& \pi \\
				t_2	& \pi	& 0
			\end{pmatrix}
			$
		\end{minipage}
		\\ \hline

		$e_7$
		&
		\begin{minipage}{0.15\linewidth}
			\centering
			$(+1,+2,+2)$
		\end{minipage}
		&
		\begin{minipage}{0.15\linewidth}
			\centering
			$(+--)$ \\
			$(+++)$
		\end{minipage}
		&
		\begin{minipage}{0.26\linewidth}
			\centering
			\drawCircle{0}{125}{125}{$x_1$}{$x_2=x_3$}{}
		\end{minipage}
		&
		\begin{minipage}{0.20\linewidth}
			\centering
			$
			\begin{pmatrix}
				0	& t_1	& t_1 \\
				t_1	& 0 	& 0 \\
				t_1	& 0		& 0
			\end{pmatrix}
			$
		\end{minipage}
		\\ \hline

		$e_8$
		&
		\begin{minipage}{0.15\linewidth}
			\centering
			$(+1,-2,-2)$
		\end{minipage}
		&
		\begin{minipage}{0.15\linewidth}
			\centering
			$(+++)$ \\
			$(+--)$
		\end{minipage}
		&
		\begin{minipage}{0.26\linewidth}
			\centering
			\drawCircle{0}{235}{235}{$x_1$}{$x_2=x_3$}{}
		\end{minipage}
		&
		\begin{minipage}{0.20\linewidth}
			\centering
			$
			\begin{pmatrix}
				0	& t_2	& t_2 \\
				t_2	& 0 	& 0 \\
				t_2	& 0		& 0
			\end{pmatrix}
			$
		\end{minipage}
		\\ \hline

		$e_9$
		&
		\begin{minipage}{0.15\linewidth}
			\centering
			$(+1,+1,+2)$
		\end{minipage}
		&
		\begin{minipage}{0.15\linewidth}
			\centering
			$(++-)$ \\
			$(+++)$
		\end{minipage}
		&
		\begin{minipage}{0.26\linewidth}
			\centering
			\drawCircle{0}{0}{115}{$x_1 = x_2$}{}{$x_3$}
		\end{minipage}
		&
		\begin{minipage}{0.20\linewidth}
		\centering
		$
		\begin{pmatrix}
			0	& 0		& t_1 \\
			0	& 0 	& t_1 \\
			t_1	& t_1	& 0
		\end{pmatrix}
		$
		\end{minipage}
		\\ \hline

		$e_{10}$
		&
		\begin{minipage}{0.15\linewidth}
			\centering
			$(+1,+1,-2)$
		\end{minipage}
		&
		\begin{minipage}{0.15\linewidth}
			\centering
			$(+++)$ \\
			$(++-)$
		\end{minipage}
		&
		\begin{minipage}{0.26\linewidth}
			\centering
			\drawCircle{0}{0}{245}{$x_1 = x_2$}{}{$x_3$}
		\end{minipage}
		&
		\begin{minipage}{0.20\linewidth}
			\centering
			$
			\begin{pmatrix}
				0	& 0		& t_2 \\
				0	& 0 	& t_2 \\
				t_2	& t_2	& 0
			\end{pmatrix}
			$
		\end{minipage}
		\\ \hline

		$e_{11}$
		&
		\begin{minipage}{0.15\linewidth}
			\centering
			$(+1,+2,+1)$
		\end{minipage}
		&
		\begin{minipage}{0.15\linewidth}
			\centering
			$(+-+)$ \\
			$(+++)$
		\end{minipage}
		&
		\begin{minipage}{0.26\linewidth}
			\centering
			\drawCircle{0}{125}{0}{$x_1=x_3$}{$x_2$}{}
		\end{minipage}
		&
		\begin{minipage}{0.20\linewidth}
			\centering
			$
			\begin{pmatrix}
				0	& t_1	& 0 \\
				t_1	& 0 	& t_1 \\
				0	& t_1	& 0
			\end{pmatrix}
			$
		\end{minipage}
		\\ \hline

		$e_{12}$
		&
		\begin{minipage}{0.15\linewidth}
			\centering
			$(+1,-2,+1)$
		\end{minipage}
		&
		\begin{minipage}{0.15\linewidth}
			\centering
			$(+++)$ \\
			$(+-+)$
		\end{minipage}
		&
		\begin{minipage}{0.26\linewidth}
			\centering
			\drawCircle{0}{235}{0}{$x_3=x_1$}{$x_2$}{}
		\end{minipage}
		&
		\begin{minipage}{0.20\linewidth}
			\centering
			$
			\begin{pmatrix}
				0	& t_2	& 0 \\
				t_2	& 0 	& t_2 \\
				0	& t_2	& 0
			\end{pmatrix}
			$
		\end{minipage}
		\\ \hline
	\end{longtable}

	\begin{longtable}{|c|c|c|c|c|}
		\caption{The 0-simplices of $\St_3(\Sp^1)$. All points are either equal or antipodal to one another.}\\
		
		\hline
		Name & $\mathfrak{c}$ & $V(\mathfrak{c})$ & $\mathbf{x} \in \Im(\Phi_{\mathfrak{c}})$ & $D_3(\mathbf{x}) = D_3(\mathbf{v}^{(1)}(\mathfrak{c}))$\\
		\hline
		\endhead

		$v_1$
		&
		\begin{minipage}{0.15\linewidth}
			\centering
			$(+1,-1,-1)$
		\end{minipage}
		&
		\begin{minipage}{0.15\linewidth}
			\centering
			$(+--)$
		\end{minipage}
		&
		\begin{minipage}{0.26\linewidth}
			\centering
			\drawCircle{0}{180}{180}{$x_1$}{$x_2=x_3$}{}
		\end{minipage}
		&
		\begin{minipage}{0.20\linewidth}
			\centering
			$
			\begin{pmatrix}
				0	& \pi	& \pi \\
				\pi	& 0 	& 0 \\
				\pi	& 0		& 0
			\end{pmatrix}
			$
		\end{minipage}
		\\ \hline

		$v_2$
		&
		\begin{minipage}{0.15\linewidth}
			\centering
			$(+1,+1,-1)$
		\end{minipage}
		&
		\begin{minipage}{0.15\linewidth}
			\centering
			$(++-)$
		\end{minipage}
		&
		\begin{minipage}{0.26\linewidth}
			\centering
			\drawCircle{0}{0}{180}{$x_1=x_2$}{}{$x_3$}
		\end{minipage}
		&
		\begin{minipage}{0.20\linewidth}
			\centering
			$
			\begin{pmatrix}
				0	& 0		& \pi \\
				0	& 0 	& \pi \\
				\pi	& \pi	& 0
			\end{pmatrix}
			$
		\end{minipage}
		\\ \hline

		$v_3$
		&
		\begin{minipage}{0.15\linewidth}
			\centering
			$(+1,-1,+1)$
		\end{minipage}
		&
		\begin{minipage}{0.15\linewidth}
			\centering
			$(+-+)$
		\end{minipage}
		&
		\begin{minipage}{0.26\linewidth}
			\centering
			\drawCircle{0}{180}{0}{$x_1 = x_3$}{$x_2$}{}
		\end{minipage}
		&
		\begin{minipage}{0.20\linewidth}
			\centering
			$
			\begin{pmatrix}
				0	& \pi	& 0 \\
				\pi	& 0 	& \pi \\
				0	& \pi	& 0
			\end{pmatrix}
			$
		\end{minipage}
		\\ \hline

		$v_4$
		&
		\begin{minipage}{0.15\linewidth}
			\centering
			$(+1,+1,+1)$
		\end{minipage}
		&
		\begin{minipage}{0.15\linewidth}
			\centering
			$(+++)$
		\end{minipage}
		&
		\begin{minipage}{0.26\linewidth}
			\centering
			\drawCircle{0}{0}{0}{$x_1 = x_2=x_3$}{}{}
		\end{minipage}
		&
		\begin{minipage}{0.20\linewidth}
			\centering
			$
			\begin{pmatrix}
				0	& 0		& 0 \\
				0	& 0 	& 0 \\
				0	& 0		& 0
			\end{pmatrix}
			$
		\end{minipage}
		\\ \hline
	\end{longtable}
\end{example}
 
\end{document}